\documentclass[11pt,letterpaper]{amsart}
\usepackage{comment}
\usepackage[latin9]{inputenc}
\usepackage{float}
\usepackage{mathtools}
\usepackage{amstext}
\usepackage{amsthm}
\usepackage{amsmath}
\usepackage{amssymb}
\usepackage{algorithm}

\usepackage{enumitem}

\usepackage[backref=page,colorlinks,citecolor=blue,bookmarks=true]{hyperref}\hypersetup{linkcolor=blue,  citecolor=blue, urlcolor=blue}

\usepackage[nobysame,abbrev,alphabetic]{amsrefs}

\usepackage{tikz}

\usepackage{nicefrac,xfrac}

\floatname{algorithm}{Algorithm}

\makeatletter

\numberwithin{equation}{section}
\numberwithin{figure}{section}
\theoremstyle{plain}
\newtheorem{thm}{\protect\theoremname}[section]
\DeclareMathOperator*{\Ex}{\mathbb{E}}
\DeclareMathOperator*{\Pro}{\mathbb{P}}

\newtheorem{cor}[thm]{Corollary}
\newtheorem*{thm*}{\protect\theoremname}
\theoremstyle{definition}
\newtheorem{problem}[thm]{\protect\problemname}
\newtheorem*{problem*}{Problem}
\theoremstyle{remark}
\newtheorem*{rem*}{\protect\remarkname}
\theoremstyle{remark}
\newtheorem{rem}[thm]{\protect\remarkname}

\theoremstyle{definition}
\newtheorem{defn}[thm]{\protect\definitionname}
\theoremstyle{plain}
\newtheorem{prop}[thm]{\protect\propositionname}
\theoremstyle{plain}
\newtheorem{fact}[thm]{\protect\factname}
\theoremstyle{definition}
\newtheorem{example}[thm]{\protect\examplename}
\theoremstyle{plain}
\newtheorem{lem}[thm]{\protect\lemmaname}
\theoremstyle{plain}
\newtheorem{claim}[thm]{Claim}
\theoremstyle{plain}
\usepackage[backref=page,colorlinks,citecolor=blue,bookmarks=true]{hyperref}\hypersetup{linkcolor=blue,  citecolor=blue, urlcolor=blue}

\usepackage[nobysame,abbrev,alphabetic]{amsrefs}
\usepackage{algorithmic}

\usepackage{microtype}

\floatname{algorithm}{Algorithm}

\let\originalleft\left
\let\originalright\right
\renewcommand{\left}{\mathopen{}\mathclose\bgroup\originalleft}
\renewcommand{\right}{\aftergroup\egroup\originalright}

\makeatother

\usepackage{babel}
\makeatletter
\addto\extrasfrench{%
   \providecommand{\fg}{\ifdim\lastskip>\z@\unskip\fi~\frqq}%
}

\makeatother
\addto\captionsenglish{\renewcommand{\definitionname}{Definition}}
\addto\captionsenglish{\renewcommand{\factname}{Fact}}
\addto\captionsenglish{\renewcommand{\lemmaname}{Lemma}}
\addto\captionsenglish{\renewcommand{\problemname}{Problem}}
\addto\captionsenglish{\renewcommand{\propositionname}{Proposition}}
\addto\captionsenglish{\renewcommand{\remarkname}{Remark}}
\addto\captionsenglish{\renewcommand{\theoremname}{Theorem}}
\addto\captionsfrench{}
\addto\captionsfrench{\renewcommand{\definitionname}{Definition}}
\addto\captionsfrench{\renewcommand{\factname}{Fait}}
\addto\captionsfrench{\renewcommand{\lemmaname}{Lemme}}
\addto\captionsfrench{\renewcommand{\problemname}{Problème}}
\addto\captionsfrench{\renewcommand{\propositionname}{Proposition}}
\addto\captionsfrench{\renewcommand{\remarkname}{Remarque}}
\addto\captionsfrench{\renewcommand{\theoremname}{Théorème}}
\providecommand{\definitionname}{Definition}
\providecommand{\factname}{Fact}
\providecommand{\lemmaname}{Lemma}
\providecommand{\problemname}{Problem}
\providecommand{\propositionname}{Proposition}
\providecommand{\remarkname}{Remark}
\providecommand{\theoremname}{Theorem}
\providecommand{\examplename}{Example}

\newcommand{\eps}{\varepsilon}

\newcommand{\Sym}{{\rm Sym}}

\newcommand{\Id}{{\rm Id}}

\newcommand{\FF}{\mathbb{F}}
\newcommand{\cA}{\mathcal{A}}

\newcommand{\cC}{\mathcal{C}}

\newcommand{\cD}{\mathcal{D}}
\newcommand{\cG}{\mathcal{G}}
\newcommand{\cX}{\mathcal{X}}
\newcommand{\cY}{\mathcal{Y}}
\newcommand{\cZ}{\mathcal{Z}}

\newcommand{\cN}{\mathcal{N}}

\title[Cocycle stability of random simplicial complexes]{Cocycle stability in permutations of random simplicial complexes}

\author[M.\ Chapman]{Michael Chapman}
\address{Michael Chapman\hfill\break
	 Institute for Advanced Study\hfill\break
	1 Einstein Dr, Princeton, NJ 08540, USA.}
\email{mchapman@ias.edu}

\author[Y.\ Peled]{Yuval Peled}
\address{Yuval Peled\hfill\break
	Einstein Institute of Mathematics\hfill\break
	The Hebrew University, Jerusalem 91904, Israel.}
\email{yuval.peled@mail.huji.ac.il}

\begin{document}

\begin{abstract}
    Finding a non-sofic hyperbolic group will resolve two major problems in geometric group theory: Are there non sofic groups? Are there non residually finite hyperbolic groups? 
    
    In this paper,  we propose a new probabilistic approach to this problem, based on the \emph{cocycle stability in permutations} of \emph{random $2$-dimensional Linial--Meshulam complexes}. Specifically, we study their cocycle stability rate, which measures how far cochains with small coboundaries are from being cocycles. 
      Our main contribution is the following: If, in a middle triangle density range, these random complexes typically have a linear cocycle stability rate, then there exists a non-sofic hyperbolic group.  Our proof method is inspired by a well known fact about the non local testability of Sipser--Spielman expander codes. 
\end{abstract}

\maketitle

\section{\textbf{Introduction}}

The study of random graphs was initiated by  Erd{\H{o}}s and R{\'e}nyi~\cite{erdHos1960evolution}. They defined the  model $G(n,p)$, where $n\in \mathbb{N}$ and $p\in [0,1]$, which now bares their names. A graph $\cG\sim G(n,p)$ has  $[n]=\{1,...,n\}$ as a vertex set, and an edge $xy$ is included in $\cG$ with probability $p$ independently of all other edges. The problems  Erd{\H{o}}s and R{\'e}nyi studied in their paper revolve around its connectedness, and in particular they have shown that the transition from disconnected almost surely to connected almost surely happens at the threshold $p=\frac{\log n}{n}$.

In the last couple of decades, the field of \emph{high dimensional combinatorics}, and in particular \emph{high dimensional expansion}, developed significantly. A seminal work which initiated much of this progress is due to Linial and Meshulam \cite{linial_meshulam2006homological}, who generalized the $G(n,p)$ model to higher dimensions.\footnote{This is one of several models for random complexes that have been studied (see \cite{kahle2016random} and the references therein).}
In the $2$-dimensional case, a simplicial complex $\cY\sim Y(n,p)$ has $[n]$ as a vertex set, a complete $1$-dimensional skeleton (namely, all possible edges), and every $2$-dimensional face (i.e., triangle) appears independently with probability $p$. 
Similar to  Erd{\H{o}}s and R{\'e}nyi, Linial and Meshulam studied the expected \emph{high dimensional connectivity} of such a random complex.
A graph is connected if and only if its $0^{\rm th}$ (reduced) cohomology with $\FF_2$-coefficients is trivial (all cohomological terms will be formally defined in Section \ref{sec:prelims}). 
Thus, it is natural to ask for which $n$ and $p$ does one expect the $1^{\rm st}$ cohomology of $\cY\sim Y(n,p)$ with $\FF_2$-coefficients to vanish. Linial and Meshulam showed that this occurs at $p=\frac{2\log n}{n}$. To that end, they have shown that the following property, later termed \emph{coboundary expansion with $\FF_2$-coefficients}, holds when $p$ is above the threshold: If a $1$-cochain $\alpha$ of $\cY$ with $\FF_2$-coefficients has a ``small'' coboundary $\delta \alpha$, then $\alpha$ is ``close'' to a $1$-coboundary. This property is seen today as a version of high dimensional expansion, and is both elusive and powerful (cf.\ \cite{gromov2010singularities,kaufman2016isoperimetric,evra2016bounded,dikstein2023coboundary,CL_part2,dikstein2024low,bafna2025quasi,bafna2025constant}).  Results of this form were later extended to cohomologies over other rings, to higher dimensions, and to quantitative bounds on the coboundary expansion~\cites{meshulam2009homological,hoffman2017threshold,kahle2021spectral,LuP,newman2023integer}.

One can define the following notion which generalizes the above: The \emph{cocycle stability rate} of a $2$-dimensinoal simplicial complex $\cY$ measures how far are $1$-cochains with coefficients in a metric group $(\Gamma,d)$  are from being cocycles,\footnote{Note that here we measure the distance from the cocycles and not the coboundaries as Linial and Meshulam did. This is more natural in our context.} as a function of the norm of their coboundary. 
When this rate is linear, the coefficient in the rate is an immediate generalization of the well studied Cheeger constant, which measures the edge expansion of a graph (cf.\ \cite{Hoory_Linial_Wigderson}). 
In \cites{CL_part1,CL_part2} Lubotzky and the first author initiated the study of the cocycle stability rate of  complexes with \emph{permutation coefficients equipped with the Hamming metric},\footnote{Dinur and Meshulam \cite{Dinur-Meshulam} studied this concept even before Lubotzky and the first author, but with respect to the discrete metric instead of the Hamming metric.} and defined the appropriate Cheeger constant $h_1(\cY,\Sym)$ in this context (see Section \ref{sec:prelims} for the precise definitions). 
In this paper, we take Linial and Meshulam's approach one step further, and study $h_1(\cY,\Sym)$ for complexes in their model.

To describe our main contribution, we need to recall the work of  Babson, Hoffman and Kahle ~\cite{babson2011fundamental}. They showed that for every constant $\eta>0$, if $p=n^{-1/2-\eta}$, then the fundamental group of $\cY\sim Y(n,p)$ is asymptotically almost surely (a.a.s.) hyperbolic and non-trivial, and that the exponent $-1/2$ is sharp. The upper bound was later improved to $p=O(n^{-1/2})$~\cite{luria2022simple}. In particular, there exists a vast \emph{mid-range} $p=n^{-1+\eta}$, $0<\eta<1/2$, in which $\cY$ a.a.s.\ has a non-trivial fundamental group and a trivial first-homology group. In this mid-range we prove the following:

\begin{thm}\label{thm:midrange}
     Let $0<\eta<\nicefrac{1}{2}$, $p=n^{-1+\eta}$ and $\cY\sim Y(n,p)$. Then, if a.a.s.\ $h_1(\cY,\Sym)= \omega(n^{-3}p^{-1})$,\footnote{We use the standard Bachmann--Landau asymptotic notation  $O,o,\Omega,\omega$ and $\Theta$.} then the fundamental group of such a complex is a.a.s.\ a non-sofic hyperbolic group.
\end{thm}
    Since every residually finite group is sofic, the above theorem opens a path towards two well known open problems: (i)
        Find a non-residually finite hyperbolic group \cite{kapovich2000equivalence}, and (ii)
         Find a non-sofic group \cites{Gromov_sofic,Weiss_Sofic}.

\begin{rem}
    Dogon \cite{dogon2023flexible} proved a beautiful result  analogous to Theorem \ref{thm:midrange} in the ``unitaries equipped with the normalized Hilbert--Schmidt norm'' setup (also known as the ``hyperlinearity'' setup). For the reader who seeks to compare the two results: His result assumes less about the random object  --- namely, that it is stable with some rate and not necessarily with a linear rate bounded away from zero --- but the conclusion is weaker as well  --- namely, the non-sofic groups\footnote{Actually, Dogon gets non-hyperlinear groups, which is a stronger conclusion than non-soficity, but our point is that his result says nothing about finding non-residually finite hyperbolic groups.} acquired by his method are not (known to be) hyperbolic ones. Also, the methods he uses are quite deep, while ours are, essentially, elementary. 
    The permutation analogue of Dogon's result, where a stability rate which is worse than linear occurs with positive probability, remains open. 
\end{rem}

Outside the mid-range, we are able to bound the cocycle Cheeger constant of $\cY\sim Y(n,p)$. First, the following theorem deals with the parameter range in which $\cY$ is typically simply-connected.

\begin{thm}\label{thm:high_range}
Fix $0<\eta<\nicefrac{1}{2}$, $p=n^{-\nicefrac{1}{2}+\eta}$ and $\cY\sim Y(n,p)$.  Then, a.a.s.,  $h_1(\cY,\Sym)\geq \frac{\eta}{42}$.
\end{thm}

Second, we are able to give close upper and lower bounds for  the cocycle Cheeger constant of $\cY\sim Y(n,p)$ if $p\ll n^{-1}$. The behavior of $\cY$ in the range was studied in \cites{kozlov2010threshold,farberDCG}, where it was shown to be a.a.s.\ collapsible to a graph and hence have no $2$-dimensional homology. The phase transitions of these properties actually occur when $p=\Theta(1/n)$ and are extensively studied (e.g., \cites{aronshtam2013collapsibility,linial_peled2016phase}).

\begin{thm}\label{thm:low_range}
    Let $2>\eta>0$, $p= n^{-1-\eta}$, and $\cY\sim Y(n,p)$. Then, a.a.s., $\eta /2\le h_1(\cY,\Sym)\le 9\eta/4+3$.
\end{thm}
We note that the lower bound holds for all $\eta>0$, but for the upper bound, if $\eta\ge 2$ then $\cY$ has no triangular faces with probability bounded away from $0$, and in such a case its Cheeger constant is infinite.

\begin{rem}
    
We note that having a cocyle Cheeger constant (in dimension $1$ with permutation coefficients) bounded away from zero is \textit{not}  a monotone property. Hence, these results do not shed light on the typical cocycle Cheeger constant in the mid-range. This is quite unfortunate, as our first result asserts that having a bounded cocycle Cheeger constant in the mid-range implies the existence of non-sofic hyperbolic groups.
\end{rem}
\subsection*{Proof ideas}

\subsubsection*{On the mid-range}

Section \ref{sec:res_fin_hyperbolic_groups} is devoted to the proof of Theorem \ref{thm:midrange}. Namely, that  $h_1(\cY,\Sym)= \omega\left( n^{-3}\cdot p^{-1}\right)$ a.a.s.\ implies the existence of   non-sofic hyperbolic groups. Our argument is inspired by the non local testability of Sipser--Spielman \cite{Sipser_Spielman_exp_codes} expander codes (cf. \cite{ben2003some}). Let us sketch the proof of this theorem, as the ideas are quite straightforward:
\begin{itemize}
    \item Fix $0<\eta<\nicefrac{1}{2}$. Assume that every hyperbolic group is sofic. Moreover, assume that $h_1(\cY,\Sym)=\omega\left( n^{-3}\cdot p^{-1}\right)$ a.a.s.\ for $\cY\sim Y(n,p)$, where $p=n^{-1+\eta}$.

    \item We will focus on pairs $\cY,\cZ$ that differ  by a single triangle. We sample them as follows: First, sample $\cY\sim Y(n,p)$. 
    Then, add to $\cY$ a single unifromly random triangle --- which we denote by $\Delta$ --- to get $\cZ$. 
    By  \cite{babson2011fundamental} and \cite{kahle2021spectral}, the fundamental groups of such a pair are a.a.s:
 $\diamond$ non-trivial;
        $\diamond$ with property (T);
        $\diamond$ hyperbolic, and thus by our assumption, sofic;
        $\diamond$ the triangle $\Delta$ is non-trivial in the fundamental group of $\cY$.
    Moreover, by our assumption, the complexes $\cY$ and $\cZ$ \textbf{both} have Cheeger constants $ \omega(n^{-3}\cdot p^{-1})>0$.
    By a standard observation due to Glebsky and Rivera (cf.\ Proposition \ref{prop:sof+stable=res_fin}), since the fundamental groups of $\cY$ and $\cZ$ are sofic and homomorphism stable\footnote{A positive cocycle Cheeger constant implies $\rho$-cocycle stability, which in turn implies $\rho$-homomorphism stability of the fundamental group.} (as in Section \ref{sec:group_soficity}), they are residually finite. 
     Lastly, by standard concentration of measure  techniques (cf.\ Lemma \ref{lem:Chernoff_bound}), the number of triangles  in $\cZ$ satisfies  $|\cZ(2)|\approx{p\binom{n}{3}}$.

    \item By the resiudal finiteness of the fundamental group of $\cY$ and the fact $\Delta$ is non-trivial in it, there is a $1$-cocycle $\alpha$ of $\cY$ such that $d_h(\alpha (\Delta),\Id)=1$. Since $\alpha$ is a cocycle of $\cY$, its defect  with respect to $\cZ$ is  $\frac{1}{|\cZ(2)|}\approx \frac{1}{p\binom{n}{3}}$. On the other hand, by property (T) of the fundamental group of $\cY$, every $1$-cocycle of $\cZ$ must be some constant distance away from $\alpha$.\footnote{In the proof of Theorem 
\ref{thm:midrange} we use local spectral expansion (Section \ref{sec:spectral_local_expansion}) of the Linial--Meshulam complexes in this range, which implies (T) for the fundamental group, and provides an easier to analyze setup.} Hence $h_1(\cZ,\Sym)=O(n^{-3}\cdot p^{-1})$, which is a contradiction. 
\end{itemize} 

\subsubsection*{On the trivial fundamental group regime}
Section \ref{sec:high_range} is devoted to the proof of Theorem \ref{thm:high_range}. Our approach is the following. Babson--Hoffman--Kahle \cite{babson2011fundamental} have shown that  the fundamental group is trivial in this regime, by proving that every triangular  path $x\to y \to z\to x$ has a filling (Definition \ref{defn:filling}) --- also known as a Van  Kampen diagram with this perimeter --- of a specific form (see Figure \ref{fig:BHK_filling}). For our result, we obtain good bounds on (i) the number of Babson--Hoffman--Kahle  fillings for every  triangular path, and (ii) the number of such fillings each triangle in the random complex belongs to. These bounds are somewhat technical, and appear in Section \ref{sec:prop_of_LM_complexes}. Given these bounds, the conclusion is quite immediate.  

\subsubsection*{On the subcritical regime}
Section \ref{sec:low_range} is devoted to the proof of Theorem \ref{thm:low_range}. In this regime, the random complex is known to typically be comprised of quite simple strongly-connected components of bounded size that are collapsible to graphs~\cite{farberDCG}. 
For the upper bound, we find a specific strongly connected-component that typically appears in $\cY$ and has the desired Cheeger constant. The more intersting part is the lower bound, where we run the collapse process in each component in reverse, and use it to transform any given cochain into a nearby cocycle.

\subsection*{Outline}
In Section \ref{sec:prelims}, we provide the required preliminaries regarding cocycle stability  with permutation coefficients of $2$-dimensional simplicial complexes.  In Section \ref{sec:prop_of_LM_complexes}, we recover some a.a.s.\ properties of Linial--Meshulam random complexes. 
Section \ref{sec:res_fin_hyperbolic_groups} is devoted to the mid range regime, where we show that linear stability implies the existence of non-sofic hyperbolic groups. Then, in Section \ref{sec:high_range} we study the high triangle density regime, where the fundamental group of the complex is trivial. Finally, Section \ref{sec:low_range} is devoted to the low triangle density regime, where the complex is collapsible to a graph.

\subsection*{Acknowledgments}
We would like to thank Alex Lubotzky for suggesting to us this problem.
Michael Chapman is supported by the National Science Foundation
under Grant No. DMS-2424441. Part of this research was conducted while he was still a Simons Junior Fellow, supported by a grant from the Simons Foundation (N. 965535). 
Yuval Peled was partially supported by the 
Israel Science Foundation grant ISF-3464/24. 
\section{\textbf{Cocycle stability with permutation coefficients}}\label{sec:prelims}

We quickly repeat many definitions and observations established in \cite{CL_part1} and \cite{CL_part2}.

\subsection{Simplicial complexes and  non-commutative cohomology}\label{sec:high_dim_cohomology_abelian_coeff}

A \emph{simplicial complex} $\cX$ is a downwards closed system of finite sets, namely, if $A\in \cX$ and $B\subseteq A$, then $B\in \cX$. One can associate a CW complex with $\cX$ by letting each set of size $d+1$ be a $d$-dimensional simplex, and gluing these simplices along their intersections. We say that $\cX$ is $d$-dimensional if the size of all of its sets is bounded by $d+1$.  The $i$-cells $\cX(i)$ are the sets of size $i+1$ in $\cX$. 

We study  non-commutative cohomology, and thus focus only on $2$-dimensional complexes.
When $d=2$, we have three types of cells --- vertices, edges and triangles. Thus, we can think of the complex $\cX$ as a graph $G(\cX)$ with triangles pasted on it. The edge $[e]=\{x,y\}$ can be oriented in two ways $xy=x\xrightarrow{e} y$ and $yx=y\xrightarrow{\bar e} x$. The triangle $[\Delta]=\{x,y,z\}$ can be oriented in $6$ different ways
$xyz,xzy,yzx,yxz,zxy,zyx$.
We interpret the orientation $xyz$ as the closed path $x\to y\to z\to x$ in the underlying graph. Given the orientation $\Delta=xyz$, we denote by $\bar\Delta$ the reverse orientation $xzy$ (note that this is indeed the same path traversed backwards). We denote by $\overrightarrow\cX(i)$ the oriented $i$-cells. We use  square brackets $[\cdot]$ to emphasize the use of an  un-oriented cell, though we may abuse notations and remove the square brackets while still referring to the un-oriented versions --- it should be clear from context. 

Let $\Gamma$ be a group, $d\colon \Gamma \times \Gamma \to \mathbb{R}_{\geq 0}$ a  bi-invariant metric on $\Gamma$ and  let  $\|\gamma\|=d(\gamma,\Id)$ for $\gamma\in \Gamma$.
  The \emph{$i$-cochains of $\cX$ with $\Gamma$  coefficients} are the anti-symmetric assignments of elements of $\Gamma$ to oriented $i$-cells. Namely,
    \[
        C^i(\cX,\Gamma)=\{\alpha\colon \overrightarrow{\cX}(i)\to \Gamma\mid \forall c\in \overrightarrow{\cX}(i)\colon \alpha(\bar c)=\alpha(c)^{-1}\}.
    \]
    Since we did not define orientations on vertices, there are no anti-symmetricity conditions on $0$-cochains. Moreover, for every $2$-cochain $\alpha\colon \overrightarrow\cX(2)\to \Gamma$ and  triangle $\{x,y,z\}$, we assume that $\alpha(xyz)$ and $\alpha(yzx)$ are conjugate.    
The \emph{coboundary maps} are defined as follows. For a $0$-cochain $\alpha\colon \cX(0)\to \Gamma$, its coboundary $\delta\alpha$ is the $1$-cochain 
\[
\forall xy\in \overrightarrow\cX(1)\ \colon\ \ \delta\alpha(xy)=\alpha(x)^{-1}\alpha(y).
\]
For a $1$-cochain $\alpha\colon \overrightarrow\cX(1)\to \Gamma$, its coboundary $\delta\alpha$ is the $2$-cochain
\[
\forall xyz\in \overrightarrow\cX(2)\ \colon\ \ \delta\alpha(xyz)=\alpha(xy)\alpha(yz)\alpha(zx).
\]
\begin{rem}\label{rem:maps_on_paths}
    Every $1$-cochain extends naturally  to paths in $G(\cX)$ as follows: If $\pi=x\xrightarrow{e_1}...\xrightarrow{e_\ell}y$, then $\alpha(\pi)=\alpha(e_1)\cdot...\cdot\alpha(e_\ell)$. In this perspective, $\delta\alpha$ is  the restriction of the extended $\alpha$ to the closed paths which are perimeters of triangles in $\cX$. 
\end{rem}

We assume from now on that any simplicial complex $\cX$ is given to us with probability distributions $\mu_i$  over its oriented $i$-cells $\overrightarrow\cX(i)$. We further assume that $\mu_i$ is uniform on all orientations of a specific cell, and hence we can think of $\mu_i$ as a probability distribution over $\cX(i)$ as well.
\begin{defn}\label{defn:descending_measures}
    The collection of probability measures $\mu_i$ on the $i$-cells of a simplicial complex $\cX$ is said to be \emph{descending} if $\mu_i(\sigma)$ is proportional to $\sum_{\tau \supset \sigma} \mu_{i+1}(\tau)$. In this case, $\mu_2$ induces $\mu_1$ and $\mu_0$.
\end{defn}

For any $\alpha$ and $\beta$ which are $i$-cochains with $\Gamma$ coefficients,  the \emph{distance} between them is 
\begin{equation}\label{eq:distance_between_cochains}
    d(\alpha,\beta)=\Ex_{x\sim \mu_i}[d
(\alpha(x),\beta(x))].
\end{equation}
Given a subset $A\subseteq C^i(\cX,\Gamma)$ and a cochain $\alpha\in C^i(\cX,\Gamma)$, we define the distance of $\alpha$ from $A$ to be 
\begin{equation}
    d(\alpha,A)=\inf\{d(\alpha,\varphi)\mid \varphi\in A\}.
\end{equation}
The \emph{norm} of an $i$-cochain with $\Gamma$ coefficients $\alpha$ is its distance to the constant identity cochain, namely 
\begin{equation}\label{eq:norm_of_cochain}
    \lVert\alpha\rVert=\Ex_{x\sim \mu_i}[\|\alpha(x)\|].
\end{equation}
An \emph{$i$-cocycle} is an $i$-cochain $\alpha$ for which $\delta\alpha$ is the constant identity function. We denote the collection of $i$-cocycles by $Z^i(\cX,\Gamma)$.
A $0$-cochain $\beta\colon \cX(0)\to \Gamma$ is said to be a \emph{$0$-coboundary} if it is constant. Namely, for every $x,y\in \cX(0)$ we have $\beta(x)=\beta(y)$. A $1$-cochain $\alpha\colon \overrightarrow\cX(1)\to \Gamma$ is said to be a \emph{$1$-coboundary} if it is in the image of the coboundary operator $\delta\colon C^0(\cX,\Gamma)\to C^1(\cX,\Gamma)$. Namely, there exists a $0$-cochain $\beta\colon \cX(0)\to \Sym(n)$ such that for every $xy\in \overrightarrow{\cX}(1)$, $\alpha(xy)=\delta\beta(xy)=\beta(x)^{-1}\beta(y)$. We denote by $B^i(\cX,\Sym)$ the collection of $i$-coboundaries of $\cX$.  Note that the only indices for which we  define $Z^i(\cX,\Gamma)$ and $B^i(\cX,\Gamma)$ are $i=0\ \textrm{or}\ 1$.
 We say that the \emph{$i^{\rm th}$ cohomology of $\cX$ with $\Gamma$ coefficients vanishes} if
 every $i$-cocycle of $\cX$  is an $i$-coboundary. 

There is a natural action of $0$-cochains of $\cX$ with $\Gamma$ coefficients on the $1$-cochains:
\begin{equation}\label{eq:action_0-coch_on_1-coch}
    \forall \alpha \colon \overrightarrow\cX(1)\to \Gamma ,\ \beta \colon \cX(0)\to \Gamma\ \colon\ \ \beta.\alpha(x\xrightarrow{e}y)=\beta(x)^{-1}\alpha(e)\beta(y).
\end{equation}
Note that for every path $\pi=x\xrightarrow{e_1}...\xrightarrow{e_\ell}y$ in $G(\cX)$, we have 
\begin{equation}\label{eq:conjugation_by_0_cochain_on_closed_path}
\beta.\alpha(\pi)=\beta(x)^{-1}\alpha(\pi)\beta(y).
\end{equation}
In particular, 
$\Vert\delta(\beta.\alpha)\Vert=\Vert \delta \alpha\Vert,$
and hence the action preserves $1$-cocycles. Furthermore, 
\begin{equation}\label{eq:char_of_dist_to_coboundaries}
    d(\alpha,B^1(\cX,\Gamma))=\inf\{\Vert\beta.\alpha\Vert \mid \beta\in C^0(\cX,\Gamma)\}.
\end{equation}

 Let $\rho$ be a \emph{rate function}, namely, a non-decreasing function $\rho\colon \mathbb{R}_{\geq0}\to \mathbb{R}_{\geq0}$ satisfying $\rho(\eps)\xrightarrow{\eps \to 0}0$.
    A  complex is said to be  \emph{$\rho$-cocycle stable in the $i^{\rm th}$ dimension with $\Gamma$ coefficients} if for every $i$-cochain $\alpha$ we have
    \begin{equation} \label{eq:cocyc_stability}
           d(\alpha,Z^i(\cX,\Gamma))\leq \rho(\Vert\delta\alpha\Vert).
    \end{equation}
    We call $\Vert \delta \alpha\Vert$ the \emph{local cocycle defect} of $\alpha$ and $ d(\alpha,Z^i(\cX,\Gamma))$ its \emph{global cocycle defect}.
      The  $i^{\rm th}$ \emph{cocycle Cheeger constant} of a complex $\cX$ with $\Gamma$ coefficients is 
    \begin{equation}\label{eq:cocyc_Cheeger_constant}
    h_i(\cX,\Gamma)=\inf\left\{\frac{\Vert\delta\alpha\Vert}{d(\alpha,Z^i(\cX,\Gamma))}\ \middle\vert\ {\alpha \in C^i(\cX,\Gamma)},\ \Vert\delta\alpha\Vert\neq 0\right\}.
    \end{equation}
    Similarly, the  $i^{\rm th}$ \emph{coboundary Cheeger constant} of a complex $\cX$ with $\Gamma$ coefficients is 
\begin{equation}\label{eq:cobound_Cheeger_constant}
    h^B_i(\cX,\Gamma)=\inf\left\{\frac{\Vert\delta\alpha\Vert}{d(\alpha,B^i(\cX,\Gamma))}\ \middle\vert\ {\alpha \in C^i(\cX,\Gamma)} \setminus B^i(\cX,\Gamma)\right\}.
    \end{equation}
    When $h_i(\cX,\Sym)>0$ we call $\cX$ an \emph{$i$-cocycle expander with $\Gamma$ coefficients}, and when $h^B_i(\cX,\Sym)>0$ we call it an \emph{$i$-coboundary expander with $\Gamma$ coefficients}. Note that having a positive $i^{\rm th}$ cocycle Cheeger constant is equivalent to having a linear cocycle stability rate in the $i^{\rm th}$ dimension.
    Furthermore,  if $h^B_i(\cX,\Gamma)>0$, then in particular the $i^{\rm th}$ cohomology of $\cX$ with $\Gamma$ coefficients vanishes. Also, if the $i^{\rm th}$ cohomology of $\cX$ with $\Gamma$ coefficients vanishes, then $h^B_i(\cX,\Gamma)=h_i(\cX,\Gamma)$.
    \begin{rem} \label{rem:omitting_mu_from_notation}
        The stability  rate in the $i^{\rm th}$ dimension $\rho$ of a complex $\cX$ (and thus also the Cheeger constants) depends on the distributions $\mu_i$ and $\mu_{i+1}$.
    \end{rem}

\subsection{Cohomology with permutation coefficients}\label{sec:cohom_perm_coef}
Though all the cohomological definitions we just listed were  defined for a general group $\Gamma$,  we focus on the case where $\Gamma$ is a finite permutation group equipped with the normalized Hamming distance. 

For a positive integer $n$, let $\Sym(n)$ be the symmetric group acting on $[n]=\{1,...,n\}$. Given  permutations $\sigma \in \Sym(n)$ and $\tau\in \Sym(N)$ where $N\geq n$,  the \emph{normalized Hamming distance (with errors)} between them is
\begin{equation}\label{eq:permutation_normalized_Hamming_distance}
d_h(\sigma,\tau)=1-\frac{|\{i\in [n]\mid \sigma(i)=\tau(i)\}|}{N}.
\end{equation}
Since the normalized Hamming distance can compare permutations of different sizes, we will study them in a collective manner. Let 
\[
\begin{split}
C^i(\cX,\Sym)&=\bigsqcup_{n=2}^{\infty} C^i(\cX,\Sym(n)),\\ Z^i(\cX,\Sym)&=\bigsqcup_{n=2}^{\infty} Z^i(\cX,\Sym(n)),\\ 
B^i(\cX,\Sym)&=\bigsqcup_{n=2}^{\infty} B^i(\cX,\Sym(n))
\end{split}
\]
be the \emph{$i$-cochains with permutation coefficients}, the \emph{$i$-cocycles with permutation coefficients} and the \emph{$i$-coboundaries with permutation coefficients} respectively. Thus, $d_h$ defines a metric on $C^i(\cX,\Sym)$ as in  \eqref{eq:distance_between_cochains}, and henceforth also a norm as in \eqref{eq:norm_of_cochain}. Furthermore, the definitions of cocycle stability and Cheeger constants extend naturally to the permutation coefficicents setup.

\subsection{Homomorphism stability and Group soficity} \label{sec:group_soficity}

In this section, we define homomorphism stability\footnote{This notion is usually referred to as \emph{pointwise flexible group stability in permutations}. See \cite{BeckerLubotzky}.} 
 as a special case of cocycle stability, and relate it to the well studied notion of \emph{sofic groups}. We use  a more general framework than the one discussed in the previous section, as we allow our $2$-dimensional complexes to be polygonal complexes and not only simplicial ones; namely, have $2$-cells which are general polygons and not only triangles. The theory is basically the same, and complete details appear in \cites{CL_part1,CL_part2}.

Let $\langle S|R\rangle$ be a finite group presentation (namely  $|S|,|R|<\infty$)  equipped with probability distributions $\mu_S$ and $\mu_R$ on the generators and relations respectively. The \emph{presentation complex} $\cX_{\langle S|R\rangle}$ is the following $2$-dimesnional CW complex whose fundamental group is isomorphic to $\langle S|R\rangle$ :\footnote{See \cite{Hatcher_Alg_Top} for more on CW complexes.} It has a single vertex $*$, and an edge $e(s)$ for every generator $s\in S$. Then, for every $r= s_1^{\eps_1}\cdot...\cdot s_\ell^{\eps_\ell}\in R$, we paste a $2$-cell along $\pi(r)=e(s_1)^{\eps_1}...e(s_\ell)^{\eps_\ell}$, where $e(s)^{-1}=\overline{e(s)}$ and $e(s)^{1}=e(s)$. The presentation $\langle S|R\rangle$ is \emph{$\rho$-homomorphism stable}  (in permutations) if $\cX_{\langle S|R\rangle}$ is $\rho$-cocycles stable in the $1^{\rm st}$ dimension with permutation coefficients, where $\mu_1=\mu_S$ and $\mu_2=\mu_R$.

 On the other hand, given a $2$-dimensional CW complex $\cX$, every choice of a spanning tree in its $1$-skeleton induces a presentation $\langle S|R\rangle$ which is isomorphic to $\pi_1(\cX,*)$.
\begin{fact}[Theorem 1.3 in \cite{CL_part1}]\label{fact:cocycle_stable_implies_hom_stable}
If $\cY$ is $\rho$-cocycle stable in the $1^{\rm st}$ dimension with permutation coefficients, then $\pi_1(\cY,*)$ (with the aforementioned presentation induced by some fixed spanning tree) is $|\overrightarrow\cY(1)|\cdot \rho$-homomorphism stable.
\end{fact}

\begin{defn}
    A group $\Gamma$ is \emph{residually finite} if for every $1\neq \gamma\in \Gamma$, there exists a finite group $F$ and a homomorphism $f\colon \Gamma\to F$ such that $f(\gamma)\neq 1$. 
\end{defn}

\begin{defn}
    A finitely presented group $\Gamma\cong \langle S|R\rangle$ is said to be \emph{sofic} if there is a sequence of functions $f_n\colon S\to \Sym(n)$,  such that,\footnote{$\langle\langle R\rangle \rangle$ is the normal subgroup generated by $R$.} 
    \[
      \forall r\in R\ \colon \ \ d_h(f_n(r),\Id)\xrightarrow{n\to \infty}0 \quad \textrm{and} \quad 
   \forall w\notin \langle\langle R\rangle \rangle\ \colon \ \ 
        \liminf(d_h(f_n(w),\Id))\geq \frac{1}{2}.
    \]
\end{defn}
\begin{problem}[Gromov \cite{Gromov_sofic}, Weiss \cite{Weiss_Sofic}] \label{prob:sofic_groups}
    Are there non-sofic groups?
\end{problem}
The following is an easy to prove, yet quite insightful, observation. 
\begin{prop}[Glebsky-Rivera \cite{GlebskyRivera}, see also \cite{ArzhantsevaPaunescu}]\label{prop:sof+stable=res_fin}
	If $\Gamma\cong \langle S|R\rangle$ is $\rho$-homomorphism stable --- with respect to any rate function $\rho$ --- and sofic, then it is residually finite. 
\end{prop}

\subsection{Covering stability}\label{sec:covering}

A \emph{covering} of $\cX$ is a combinatorial map $f\colon \cY\to \cX$ which is a topological covering (cf.  Chapter 1.3 in \cite{Hatcher_Alg_Top}, and \cites{Dinur-Meshulam,CL_part1}). As long as $\cX$ is connected, the \emph{degree} of the covering is well defined and is equal to $|f^{-1}(x)|$ for any point $x\in \cX$. If the degree of the covering $\cY$ is $n\in \mathbb{N}$,  we call it an $n$-covering of $\cX$. 

    Let $\cX$ be a connected simplicial complex. There is a one to one correspondence between $n$-coverings of $G(\cX)$ and orbits of $1$-cochains of $\cX$ with $\Sym(n)$ coefficients under the action described in \eqref{eq:action_0-coch_on_1-coch}.
For every $1$-cochain $\alpha\colon \overrightarrow\cX(1)\to \Sym(n)$, the corresponding covering $f\colon \cG \to G(\cX)$ is defined to be:
    \begin{itemize}
    \item $\cG(0)=\cX(0)\times [n]\quad,\quad\overrightarrow\cG(1)=\overrightarrow\cX(1)\times [n]$;
    \item $\forall x\xrightarrow{e}y\in \overrightarrow\cX(1),i\in [n]\ \colon \ \ \tau(e,i)=(y,i),\ \iota(e,i)=(x,\alpha(e).i),\ \overline{(e,i)}=(\bar e,\alpha(e).i);$\footnote{Note that the permutation $\alpha(e)$ tells us how the fiber over the terminal point $y$ is mapped to the fiber over the origin point $x$ and not the other way around. This is because of our choice of left actions.}
    \item $\forall x\in \cX(0),e\in \overrightarrow\cX(1),i\in[n]\ \colon \ \ f(x,i)=x,\quad f(e,i)=e.$
\end{itemize}
On the other hand, if $f\colon \cG\to G(\cX)$ is an $n$-covering, then one can construct a $1$-cochain $\alpha\colon \overrightarrow\cX(1)\to \Sym(n)$ as follows:
     For every $x\in \cX(0)$, $|f^{-1}(x)|=n$. Hence we can label the vertices of $f^{-1}(x)$ by $\{(x,i)\}_{i=1}^n$. Note that for each vertex there are $n!$ ways of choosing these labels. Now, for every $e\in \overrightarrow\cX(1)$ and $e'\in f^{-1}(e)$ define $e'=(e,i)$ if $\tau(e')=(y,i)$.  Then, $\alpha(e).i$ is the second coordinate of $\iota(e,i)$. The different choices of labeling for the fibers $f^{-1}(x)$ would give rise to different $1$-cochains \textbf{in the same orbit} of the action of $0$-cochains.
     This correspondence sends $1$-cocycles to  $n$-coverings of the complex $\cX$. Futhermore, $1$-coboundaries are in correspondence with disjoint unions of the base complex $\cX$. Lastly, this correspondence translates $\rho$-cocycle stability to a topological robustness of coverings --- if for a covering $\cY$  of the underlying graph $G(\cX)$ most polygons in $\cX$ lift to closed paths in $\cY$, then $\cY$ is close (in an appropriate metric on graphs) to  (the $1$-skeleton of) a covering of $\cX$.   
     
\begin{lem}\label{lem:hamming_dist_cochains_and_epxansion_in_covering}
    Let $\cY$ be a $2$-dimensional simplicial complex. Let $\alpha\colon \overrightarrow\cY(1) \to \Sym(n)$ and $\beta\colon  \overrightarrow\cY(1) \to \Sym(N)$ be two $1$-cochains of $\cY$, where $N\geq n$. Let $\alpha\times\beta\colon \overrightarrow\cY(1) \to \Sym(n\times N)$ be the cochain satisfying $\alpha\times \beta (e).(i,j)=(\alpha(e).i,\beta(e).j)$. Let $\cA$ be the covering of $G(\cY)$ associated with $\alpha\times \beta$ and equipped with the uniform lift of $\mu_1$ and $\mu_0$.  Let $D=\{(y,i,i)\mid y\in\cY(0),i\in [n]\}\subseteq \cA(0)$ be the ``diagonal'', and $f\colon \cA(0)\to \FF_2$ its indicator function. Then,
    \[
        d_h(\alpha,\beta)=\frac{n}{N}\cdot \frac{\Vert \delta f\Vert}{2\mu_0(D)}+1-\frac{n}{N},
    \]
    and as $\nicefrac{\Vert \delta f\Vert}{2\mu_0(D)}\leq 1$, it implies
    \[
     d_h(\alpha,\beta)\geq \frac{\Vert \delta f\Vert}{2\mu_0(D)}.
    \]
\end{lem}

\begin{proof}
Recall that 
\[
\begin{split}
    d_h(\alpha,\beta)&=\Pro_{\substack{e\sim \mu_1\\i\in[N]}}[\alpha(e).i\neq\beta(e).i]\\
    &=\frac{n}{N}\cdot \Pro_{\substack{e\sim \mu_1\\i\in[n]}}[\alpha(e).i\neq\beta(e).i]+\frac{N-n}{N}.
\end{split}
\]
Thus,
\[
d_h(\alpha,\beta)-1+\frac{n}{N}=\frac{n}{N}\Pro_{\substack{e\sim \mu_1\\i\in[n]}}[\alpha(e).i\neq\beta(e).i].
\]
        By the definition of the uniform lift of $\mu_0$ to $\cA$, $\mu_0(D)=\frac{1}{N}$, since $D$ contains for every $y\in \cY(0)$ exactly $\frac{n}{N\cdot n}$ of the fiber above it. By the definition of the uniform lift of $\mu_1$ to $\cA$, 
    \[
    \begin{split}
         \Vert \delta f\Vert &=\Pro_{e\sim \mu_1}\Pro_{(i,j)\in [n]\times [N]}[f(\iota(e,i,j))\neq f(\tau(e,i,j))]\\
         &=\frac{1}{nN}\sum_{\substack{e\in \overrightarrow\cY(1)\\i\in [n],j\in [N]}}\mu_1(e){\bf 1}_{f(\iota(e),\alpha(e).i,\beta(e).j)\neq f(\tau(e),i,j)}\\
         &=(\diamond).
    \end{split}
    \]
    Now, if $i= j$, then $(\tau(e),i,j)\in D$, and ${\bf 1}_{f(\iota(e),\alpha(e).i,\beta(e).j)\neq f(\tau(e),i,j)}=1$ if and only if $\alpha(e).i\neq \beta(e).i$. 
    On the other hand, if $i\neq j$, then ${\bf 1}_{f(\iota(e),\alpha(e).i,\beta(e).j)\neq f(\tau(e),i,j)}=1$ if and only if $\alpha(e).i = \beta(e).j$. Combining both cases, for every $i$ and $e$,  the sum over $j$'s contributes either $0$, when $\alpha(e).i=\beta(e).i$, or $2$, when $\alpha(e).i\neq\beta(e).i$. Hence,
    \[
    \begin{split}
        (\diamond)=\frac{1}{nN}\sum_{\substack{e\in \overrightarrow\cY(1)\\i\in [n]}}2\mu_1(e){\bf 1}_{\alpha(e).i\neq \beta(e).i}.
    \end{split}
    \]
    Therefore, 
    \[
    \begin{split}
    \frac{\Vert \delta f \Vert}{\mu_0(D)}&=2\Pro_{\substack{e\sim \mu_1\\ i\in[n]}}[\alpha(e).i\neq \beta(e).i]\\
    \end{split}
    \]
    and we conclude.
\end{proof}

\begin{rem}
Though Lemma \ref{lem:hamming_dist_cochains_and_epxansion_in_covering} is  technical, it is quite fundamental. In some sense, it should have appeared in \cites{CL_part1,CL_part2} beforehand, as an observation. The level technique which was used in the proof of Proposition 4.3 in \cite{CL_part2}  is a special case of this Lemma. 

This idea of comparing cochains using a common diagonal action is similar to the way one compares representations by analysing their tensor product. The cochain $\alpha\times \beta$ is essentially the tensor product  of $\alpha$ and $\beta$. See Proposition 2.4 in \cite{BeckerLubotzky} for a similar result.
\end{rem}

\subsection{Local spectral expansion}\label{sec:spectral_local_expansion}
This section discusses local spectral expansion, which using a local to global result a la Garland \cite{garland1973p}, specifically in the form known as Oppenheim's trickling down theorem \cite{oppenheim2018local}, 
 implies global spectral expansion.
 We follow  \cite{Dikstein_lecture_notes_trickling_down} in our presentation --- see also Appendix A of \cite{harsha2022note}, and the original paper by Oppenheim for the uniform case \cite{oppenheim2018local}. 

Let $\cX$ be a connected $2$-dimensional simplicial complex. Assume $\mu_2$ is a fully supported distribution on the triangles of $\cX$. Furthermore, assume $\mu_1$ and $\mu_0$ are descendent (as in Definition \ref{defn:descending_measures}) from $\mu_2$, and are fully supported. This implies in particular that the complex $\cX$ is \emph{pure}: Every edge (and vertex) participates in some triangle. The link of a cell  $\sigma\in \cX$ is the simplicial complex
\begin{equation}\label{eq:defn_link}
    \cX_\sigma=\{\tau\in \cX\mid \tau\cup \sigma\in \cX, \tau \cap \sigma=\emptyset\}.
\end{equation}
As this paper focuses on  two dimensional complexes, mostly vertex links will be used, which are themselves graphs (and sometimes edge links will be used, which are collections of vertices). The links inherit a distribution on their edges as follows: For every vertex $v$ and edge $e\in \cX_v$, we have $\mu_{v,1}(e)=\frac{\mu_2(v\cup e)}{\sum_{e'\in \cX_v(1)}\mu_2(v \cup e')}$. The measure on vertices of $\cX_v$ descends from $\mu_{v,1}$. Given a simplical complex $\cX$, we can define an adjacency operator $A$ on the complex valued functions on its vertices: Given $f\colon \cX(0)\to \mathbb{C}$, $$Af(v)=\Ex_{u\sim \mu_{v,0}}[f(u)],$$
namely, we sample an edge $vu$ according to $\mu_1$ conditioned on the edge containing $v$. Furthermore, we choose the following inner product on $\mathbb{C}^{\cX(0)}$:
\[
\langle f,g\rangle=\Ex_{v\sim \mu_0}[\overline{f(v)}g(v)].
\]
The adjacency operator $A$ has the following properties:
    $\bullet$ The constant functions are eigenfunctions of it with eigenvalue $1$.
    $\bullet$ It is self adjoint with respect to the chosen inner product.

\begin{defn}\label{defn:spectral_expansion}
    We say that $\cX$ is a \emph{$\lambda$-spectral expander} if the second eigenvalue of $A$ is bounded by $\lambda$ from above. We say that $\cX$ is a \emph{$\lambda$-local spectral expander} if the links $\cX_v$ are $\lambda$-spectral expanders for all $v\in \cX(0)$.
\end{defn}

\begin{fact}[Oppenheim's trickling down theorem, \cites{Dikstein_lecture_notes_trickling_down,oppenheim2018local}]\label{fact:trickle}
   Let $\cX$ be a pure, connected, $2$-dimensional simplicial complex with fully supported descendent measures on its cells.  Assume also that $\cX$ is a $\lambda$-local spectral expander for some $\lambda<\frac{1}{2}$ --- namely, that for each vertex $v\in \cX(0)$, the link $\cX_v$ is a $\lambda$-spectral expander. Then $\cX$ is a $\frac{\lambda
}{1-\lambda}$-spectral expander.
\end{fact}

\begin{prop}[Weighted Cheeger inequalities: The lower bound. See Appendix A in \cite{CL_part2}]\label{fact:weighted_Cheegr_lower_bound}
    If $\cX$ is a $\lambda$-spectral expander, then $h_0(\cX,\FF_2)\geq 1-\lambda$.
\end{prop}

\begin{cor}\label{cor:local_exp_implies_exp_of_coverings}
    Let $\cX$ be a $\lambda$-local spectral expander for some $\lambda<\frac{1}{2}$. Then, for every connected covering $\cY$ of $\cX$, we have $h_0(\cY,\FF_2)\geq \frac{1-2\lambda}{1-\lambda}$.
\end{cor}
\begin{proof}
    Since links are preserved by coverings, $\cY$ is a $\lambda$-local spectral expander. Since $\cY$ is connected as well, it satisfies the conditions of Fact \ref{fact:trickle}, and thus it is a $\frac{\lambda}{1-\lambda}$-spectral expander. By Proposition \ref{fact:weighted_Cheegr_lower_bound}, we deduce that $h_0(\cY,\FF_2)\geq 1-\frac{\lambda}{1-\lambda}=\frac{1-2\lambda}{1-\lambda}$.
\end{proof}

\section{\textbf{Structural properties of Linial--Meshulam complexes}}\label{sec:prop_of_LM_complexes}

This section collects various a.a.s.\ properties of complexes sampled in the Linial--Meshulam model.
We will repeatedly use the following concentration of measure lemma. It is a version of  the well known Chernoff bound \cite{chernoff1952measure}:

\begin{lem}[Chernoff bound, cf.\ Theorem 3.3 in \cite{Chernoff_ODonnell} or Theorems 4.4 and 4.5 in \cite{mitzenmacher2017probability}]\label{lem:Chernoff_bound} Let $k$ be a positive integer and $M>0$ a constant.
    Let $\{X_i\}_{i=1}^k$ be a collection of  independent and identically distributed random variables, where $0\leq X_i\leq M$ and $\Ex[X_i]=\mu$.  Then, for every $0\leq \eps\leq 1$, we have $$\Pro\left[\left|\sum X_i-k\mu\right|> \eps k\mu\right]\leq 2e^{-\eps^2k\mu/3M}.$$
\end{lem}

\begin{cor}[Basic properties of $Y(n,p)$]\label{cor:basic_prop}
Let $\eps,\eta>0$. The following are immediate consequences of Lemma \ref{lem:Chernoff_bound}:
    \begin{enumerate}[label=\textcolor{black}{\arabic*.}, ref=\arabic*.]
        \item \textbf{Number of triangles in $\cY$}: 
   If $p=\omega(n^{-3})$ and $\cY\sim Y(n,p)$, then a.a.s
  \[
\left\vert  \nicefrac{|\cY(2)|}{p\binom{n}{3}}-1\right\vert\leq \eps.
  \]
  \item \textbf{Edge degree in $\cY$}:  Recall that for a $2$-dimensional complex $\cY$ and an edge $xy\in \cY(1)$, the link $\cY_{xy}$ is $0$-dimensional, and it consists of the set of vertices $z\in \cY(0)$  for which $xyz\in \cY(2)$, namely $\cY_{xy}=\{z\in \cY(0) \mid xyz\in \cY(2)\}$. Recall also that for $\cY\sim Y(n,p)$, its $1$-skeleton is the complete graph on $[n]$, and thus every $x\neq y\in [n]$ form an edge.\label{clause:2_basic_prop_Ynp} 
  
Now, if $p=n^{-1+\eta}$ and $\cY\sim Y(n,p)$, then a.a.s.
    \[
\forall x\neq y\in [n]\ \colon \ \ \left\vert\nicefrac{|\cY_{xy}|}{pn}-1\right\vert\leq \eps.
    \]
    \item \textbf{The links as  Erd{\H{o}}s--R{\'e}nyi  graphs}:  It is immediate from their definitions, that for $x\in [n]$ and $\cY\sim Y(n,p)$, the link $\cY_x$ is distributed as a $G(n-1,p)$ graph. In particular, from the previous clause, one can deduce that the links in the $p=n^{-1+\eta}$ regime are ``essentially regular'', namely every vertex has approximately $np$ many neighbors.
    \item \textbf{The descendant $\mu_1$ is close to the uniform distribution}
From the first two clauses we also deduce the following. Let $p=n^{-1+\eta}$ and $\mu_2$ be the uniform distribution on the triangles of $\cY\sim Y(n,p)$. Then, a.a.s,
    \[
    \forall x\neq y\in [n]\ \colon\ \ \left\vert \mu_1(xy)\cdot \binom{n}{2}-1\right \vert\leq \eps,
    \]
    where $\mu_1$ is the distribution descendant from $\mu_2$ (Definition \ref{defn:descending_measures}).
    \end{enumerate}
\end{cor}

\begin{defn}\label{defn:balls_and_spheres}
    Given a vertex $x$ in a graph, let $B_\ell(x)$ be the ball of radius $\ell$ around it, namely all vertices who have a path of length at most $\ell$ to $x$. The sphere of radius $\ell$ around $x$, which we denote by $S_\ell(x)$, is the difference between $B_\ell(x)$ and $B_{\ell-1}(x)$, namely all vertices whose shortest path to $x$ is of length $\ell$. Finally, let $US_\ell(x)$ be the unique sphere of radius $\ell$ around $x$, namely all vertices in $S_\ell(x)$ that have a \emph{single} path of length $\ell$ to $x$. We denote their sizes by $b_\ell(x)=|B_\ell(x)|, s_\ell(x)=|S_\ell(x)|$ and $us_\ell(x)=|US_\ell(x)|$.
\end{defn}

\begin{claim}[Spheres of small radius in Erd{\H{o}}s--R{\'e}nyi  graphs]\label{claim:expansion_Gnp}
    Let $0<\eps,\eta<1$, $p=n^{-1+\eta}$ and $\cX\sim G(n,p)$. Let $x\in [n]$ be a fixed vertex, and let $\ell< \nicefrac{1}{\eta}$ be a positive integer. Then, 
    \[
\Pro[|\nicefrac{s_\ell(x)}{(pn)^\ell}-1|>  \eps]\leq 4ne^{-\eps^2\cdot 2^{-2/\eta}\cdot np/24}\ ,
    \]
    where $s_\ell(x)$ is the size of the sphere $S_\ell(x)$ of radius $\ell$ around $x$, as in Definition \ref{defn:balls_and_spheres}.
\end{claim}

\begin{proof}
The idea is to use a Breadth First Search (BFS) algorithm on a $ G(n,p)$ random graph, which induces a random process that we can analyze (cf.\ Section 11.5 of \cite{alon2016probabilistic}). 
Recall that in a BFS algorithm, one starts with a fixed vertex, in our case $x$, and an empty queue. Throughout the algorithm, the vertices can be in three states: \emph{Neutral}, which means it has not entered the queue yet; \emph{On}, which means it is currently in the queue; \emph{Off}, which means it was popped out of the queue. Hence, before the algorithm starts, all vertices are in the Neutral state. In the first step, $x$ is put into the queue (which moves it from the Neutral to the On state), and a $0$ is appended to it (which marks its distance from $x$). From that point on, at each step, we pop out the first vertex in the queue $y$ with its appended integer $r$, and add to (the end of) the queue all Neutral neighbors of $y$ in the graph with  $r+1$ appended to them. The algorithm usually halts when the queue is empty, but in our case we will halt it once the popped out  vertex has an appended integer $\ell$. At this point, all non-Neutral vertices consist of the ball of radius $\ell$ around $x$, and the appended ineteger to each of them marks their distance from $x$. In particular, the vertices in the On state when the algorithm halts consist of the sphere of radius $\ell$. 

We now associate random variables with various parameters along the BFS algorithm. First, let $\cY_t$ be the number of vertices in the queue (namely in the On state) at time step $t$, and let $\cY_0=1$, indicating that $x$ is the only vertex in the queue in the beginning of the process. Let $\cZ_t$ be the number of vertices added to the  queue in time $t$,  let $\cN_t$ be the number of Neutral vertices at time $t$, and let $\cD_t$ be the degree of the vertex popped in time $t$. Then, $\cN_0=n-1$, and the random variables satisfy the recursive relations $\cY_t=\cY_{t-1}-1+\cZ_t$ and $\cN_t=\cN_{t-1}-\cZ_t$, as well as the obvious inequality $\cD_t\geq \cZ_t$. 

Note that, if for every $t$ we have that $a\leq \cZ_t\leq b$ for some $a,b>0$, then $a^\ell\leq s_\ell(x)\leq b^{\ell}$. To see this, let us denote by $y_t$ the vertex that was popped in time $t$, and note that from the properties of the BFS algorithm we have $s_i(x)=\sum_{y_t\in S_{i-1}(x)}  \cZ_t$. Combined with the assumed bound on $\cZ_t$, we deduce that $s_{i-1}(x)a\leq s_i(x)\leq s_{i-1}(x)b$. As $s_0(x)=1$, we deduce this observation.
Hence, if we can show that
\begin{equation}\label{eq:main_consequence_sphere_bound}
    \Pro[\forall t\colon (1-\nicefrac{\eps}{2^{\ell}})np \leq  \cZ_t\leq (1+\nicefrac{\eps}{2^{\ell}})np]\geq 1-4ne^{-\eps^2\cdot 2^{-2/\eta}\cdot np/24}\ ,
\end{equation}
then the Claim is deduced, as $(1+\nicefrac{\eps}{2^{\ell}})^\ell\leq 1+\eps$ and $(1-\nicefrac{\eps}{2^{\ell}})^\ell\geq 1-\eps$. To the end of proving \eqref{eq:main_consequence_sphere_bound}, let us notice that 
\begin{align}
    &1- \Pro[\forall t\colon (1-\nicefrac{\eps}{2^{\ell}})np \leq  \cZ_t\leq (1+\nicefrac{\eps}{2^{\ell}})np]=\\
    &\Pro[\exists t\colon \cZ_t>(1+\nicefrac{\eps}{2^\ell})np]\label{eq:reduction_of_3.1}
    +\Pro[\forall t\colon \cZ_t\leq (1+\nicefrac{\eps}{2^\ell})np\  \land\  \exists t'\colon \cZ_{t'}< (1-\nicefrac{\eps}{2^\ell})np]\ .\notag
\end{align}
So, we deduce the claim by bounding each of the summands on the right hand side of the above equation.

By applying Chrenoff (Lemma \ref{lem:Chernoff_bound}) and a union bound, and using the inequality $\cD_t\geq \cZ_t$, we can deduce that 
\begin{equation}\label{eq:upper_bound_degree}
\begin{split}
    \Pro[\exists t\colon \cZ_t>(1+\nicefrac{\eps}{2^\ell})np]&\leq \Pro[\exists t\colon \cD_t\geq (1+\nicefrac{\eps}{2^\ell})np]\\
    &\leq 2ne^{-\eps^2\cdot np/{3\cdot2^{2\ell}}}\\
    &\underset{\substack{\ell\leq \nicefrac{1}{\eta}\\ 3\le 24}}{\leq} 2ne^{-\eps^2\cdot 2^{-2/\eta}\cdot np/{24}}\ .
\end{split}
\end{equation}
Let us assume that the complement of this event occurred, namely $(1+\nicefrac{\eps}{2^\ell})np$ is an upper bound on the degree of all vertices in the graph.  Hence, by our previous analysis, for every $i\leq \ell$ we have  $s_i(x)\leq (1+\eps)(np)^i$. 
The size of the ball $b_\ell(x)$ of radius $\ell$ around $x$ can thus be bounded by
\[
b_\ell(x)=\sum_{i=0}^\ell s_i(x)\leq (1+\frac{\ell}{np})(1+\eps)(np)^\ell.
\]
As $\ell$ is a fixed positive integer, $\frac{\ell}{np}\leq \frac{\eps}{2}$ for some large enough $n$. In addition, as $\ell<\nicefrac{1}{\eta}$, $(np)^{\ell}=n^{\ell \eta}$  is smaller than $Cn$ for every constant $C$ and large enough $n$. Thus, for large enough $n$, 
\[
 b_\ell(x)\leq (1+2\eps)n^{\ell \eta}\leq \frac{\eps}{2\cdot 2^{\nicefrac{1}{\eta}}} \cdot n \ .
\]
When the BFS algorithm terminates (in our case), the non-Neutral vertices are exactly those from the ball of radius $\ell$ around $x$, and thus the number of Neutral vertices ends up being  $n-b_\ell(x)$. Hence, throughout the BFS algorithm we have $\cN_t\geq n-b_\ell(x)\geq (1-\frac{\eps}{2\cdot 2^{\nicefrac{1}{\eta}}})n$. As our graph is from $G(n,p)$, every edge between the recently popped $y$ and the currently Neutral vertices appears independently with probability $p$, and thus $\cZ_t$ is distributed binomially as  $Bin(\cN_{t-1},p)$. Therefore, by Chernoff and a union bound, we have that
\begin{equation}\label{eq:lower_bound_degree_out_in_BFS}
\begin{split}
    \Pro[\forall t\colon \cZ_t\leq (1+\nicefrac{\eps}{2^\ell})np\  \land\  \exists t'\colon \cZ_{t'}\leq (1-\nicefrac{\eps}{2^{\nicefrac{1}{\eta}}})np]
    &\leq 2ne^{-\left(\frac{\eps}{2\cdot 2^{\nicefrac{1}{\eta}}} \right)^2\cdot\overbrace{(1-\frac{\eps}{2\cdot 2^{\nicefrac{1}{\eta}}})}^{\geq \nicefrac{1}{2}} np/3}\\
    &\leq 2ne^{-\eps^2\cdot 2^{-2/\eta}\cdot np/24}\ .
\end{split}
\end{equation}
The combination of \eqref{eq:upper_bound_degree} and \eqref{eq:lower_bound_degree_out_in_BFS} provides the appropriate upper bounds on the summands in \eqref{eq:reduction_of_3.1}. Therefore,  \eqref{eq:main_consequence_sphere_bound} is deduced and so is the claim itself.
\end{proof}

\begin{rem}\label{rem:unique_sphere_bound}
    Recall the notion of the unique sphere around $x$ from Definition \ref{defn:balls_and_spheres}.  A bound similar to that of Claim \ref{claim:expansion_Gnp} holds for $us_\ell(x)$ instead of $s_\ell(x)$, namely under the same setup 
     \[
\Pro[|\nicefrac{us_\ell(x)}{(pn)^\ell}-1|>  4\eps]\leq 4ne^{-\eps^2\cdot 2^{-2/\eta}\cdot np/24}\ .
    \]
    This is because, on the one hand, $us_\ell(x)\leq s_\ell(x)$. On the other hand, if for every $y\in S_\ell(x)$ we let $w(y)$ be the number of paths of length $\ell$ from $x$ to $y$, then 
    \[
s_\ell(x)-us_\ell(x)\leq \sum_{y\in S_\ell(x)} w(y)-1 \quad \textrm{and}\quad \sum_{y\in S_\ell(x)} w(y)\leq b_\ell(x).
    \]
    Combining the two inequalities, we deduce that  $s_\ell(x)-us_\ell(x)\leq b_\ell(x)-s_\ell(x)$.  But the proof of Claim \ref{claim:expansion_Gnp} shows that 
    \[
    \underbrace{b_\ell(x)}_{\leq (1+2\eps)(np)^\ell}-\underbrace{s_\ell(x)}_{\geq (1-\eps)(np)^{\ell}}\leq 3\eps (pn)^\ell
    \]
    with probability of at least $1-4ne^{-\eps^2\cdot 2^{-2/\eta}\cdot np/24}$. Thus, 
    \[
    us_\ell(x)= s_\ell(x)- (s_\ell(x)-us_\ell(x))\geq (1-\eps)(pn)^\ell- 3\eps(pn)^\ell\geq (1-4\eps)(pn)^\ell    
    \]
    with the same probability and we are done. 
\end{rem}
\begin{claim}\label{claim:intersection_of_independent_sets}
    Let $0<\alpha,\beta<1, \eps>0$ be constants. If two sets $A$ and $B$ of sizes 
    \[
     (1-\eps)n^\alpha\leq |A|\leq (1+\eps)n^\alpha\ ,\ (1-\eps)n^\beta\leq |B|\leq (1+\eps)n^\beta\ ,
    \]
 are sampled uniformly from $[n]$, then we have that 
 \[
\Pro[\nicefrac{|A\cap B|}{n^{\alpha+\beta-1}}-1|\leq 9\eps]\leq 4e^{-\eps^2 n^{\alpha+\beta-1}/6}\ .
 \]
 In other words, the probability that $|A\cap B|$ is not $(1+o(1))n^{\alpha+\beta-1}$ is exponentially small (in $n$) given that $\alpha+\beta>1$.
\end{claim}
\begin{proof}
The proof is very similar in structure to that of Claim \ref{claim:expansion_Gnp}. We can think  of $B$ as being fixed, and then revealing the elements of $A$ one by one. In that case, if after $t$ steps $k$ of the revealed elements of $A$ were from $B$ and the other $t-k$ were from its complement, then the probability that the $t+1$ element is from $B$ is 
\[
\frac{|B|-k}{n-|B|-(t-k)}\leq \frac{|B|}{n-|B|-|A|}=\underbrace{\frac{|B|}{n}}_{\leq (1+\eps)n^{\beta-1}}\cdot \underbrace{\frac{n}{n-|B|-|A|}}_{\underset{n\to \infty}{\leq}1+\eps}  \leq (1+3\eps)n^{\beta-1}\ .
\]
Hence, $|A\cap B|$ is smaller than a random variable distributed as $Bin((1+\eps)n^\alpha,(1+3\eps)n^{\beta-1})$. Therefore, using this upper bound and Chernoff (Lemma \ref{lem:Chernoff_bound}),  we deduce that 
\[
\Pro[|A\cap B|\geq (1+9\eps)n^{\alpha+\beta-1}]\leq 2e^{-\eps^2 n^{\alpha+\beta-1}/3}.
\]
Therefore, we can assume that $|A\cap B|\leq (1+9\eps)n^{\alpha+\beta-1}$.  Under this assumption, throughout the process of revealing $A$, as the parameter $k$ is bounded by $|A\cap B|$, we have that the probability that the $t+1$ element of $A$ is in $B$ is
\[
\begin{split}
\frac{|B|-k}{n-|B|-(t-k)}&\geq \frac{(1-\eps)n^\beta-(1+9\eps)n^{\alpha+\beta-1}}{n}\\
&=n^{\beta-1}\underbrace{(1-\eps-(1+9\eps)n^{\alpha-1})}_{\underset{n\to \infty}{\geq}1-2\eps }\\
&\geq (1-2\eps)n^{\beta-1}\ .
\end{split}
\]
Hence, by this upper bound and Chernoff we can deduce that 
\[
\Pro[|A\cap B|\leq (1-5\eps)n^{\alpha+\beta-1}\mid |A\cap B|\leq (1+9\eps)n^{\alpha+\beta-1}]\leq 2e^{-\eps^2n^{\alpha+\beta-1}/6},
\]
which finishes the proof.
\end{proof}

\subsection*{BHK Fillings} 

\begin{defn}[Fillings, cf.\ \cite{babson2011fundamental}]\label{defn:filling}
    We say that a closed path $\pi$ in $\cX$ can be \emph{filled} by triangles $\Delta_1,...,\Delta_k\in \cX(2)$ if one can draw a Van Kampen diagram whose perimeter is $\pi$ and its cells are the aforementioned triangles. Namely, $\pi$ is a product of conjugates of $\Delta_1,...,\Delta_k$ in the fundamental group of $\cX$ --- which is equivalent to the existence of  paths $\sigma_1,...,\sigma_k$ in $G(\cX)$ such that
    \[
        \pi=\prod_{i=1}^k \sigma_i \Delta_i \bar \sigma_i.
    \]
    We call $k$ the \emph{volume} of the filling.

Given a triplet $x,y,z\in \cX(0)$ and a positive integer $\ell$, a length $\ell$ Babson--Hoffman--Kahle  filling with (edge) base $xy$ of the path $\pi=xyz$, or just $\ell$-BHK filling for short, is of the following form: There exists a length $\ell$ path $z=a_0\to a_1 \to ...\to a_\ell$ such that 
\[
xya_\ell\in \cX(2)\quad{\rm and}\quad \forall 1\leq i\leq \ell\ \colon\ \ a_{i-1}a_ix\ ,\ a_{i-1}a_iy\in \cX(2)\ .
\]
We call the triangle $xya_\ell$ the  (triangle) \emph{base} of the filling, while the triangles $a_{i-1}a_ix$ and $a_{i-1}a_iy$ are said to be of \emph{type $i$}. 
An $\ell$-BHK filling is said to be \emph{good} if the path $z=a_0\to...\to a_\ell$ is the only path of length $\leq k$ between $z$ and $a_\ell$, namely, it is the only $r$-BHK filling of $xyz$ with (triangle) base  $xya_{\ell}$ with $r\leq \ell$. 
\begin{figure}
\begin{center}
    \begin{tikzpicture}[scale=1.5]
    \node [draw, color=black, shape=circle] (x) at (-4,0){$x$};
    \node [draw, color=black, shape=circle] (y) at (4,0){$y$};
    \node [draw, color=black, shape=circle] (z) at (0,6){$z=a_0$};
    \node [draw, color=black, shape=circle] (w) at (0,1){$a_\ell$};
    \node [draw, color=black, shape=circle] (w1) at (0,2){$a_{\ell-1}$};
    \node [draw, color=black, shape=circle] (w-1) at (0,4.5){$a_{1}$};
     \draw[purple, -, solid] (x)--(y) node[midway,above]{};
\draw[purple, -, solid] (y)--(w) ;
\draw[purple, -, solid] (y)--(z) ;
\draw[purple, -, solid] (x)--(w) ;
\draw[purple, -, solid] (w1)--(w) ;
\draw[purple, -, solid] (w1)--(x) ;
\draw[purple, -, solid] (w1)--(y) ;
\draw[purple, -, solid] (w-1)--(x) ;
\draw[purple, -, solid] (w-1)--(y) ;
\draw[purple, -, solid] (z)--(w-1) ;
\draw[purple, -, dashed] (w1)--(w-1) ;
\draw[purple, -, solid] (z)--(x) ;
    \end{tikzpicture}
    \caption[]{An $\ell$-BHK filling of $xyz$. Compare to Figure 1 in \cite{babson2011fundamental}.}\label{fig:BHK_filling}
    \end{center}
    \end{figure}
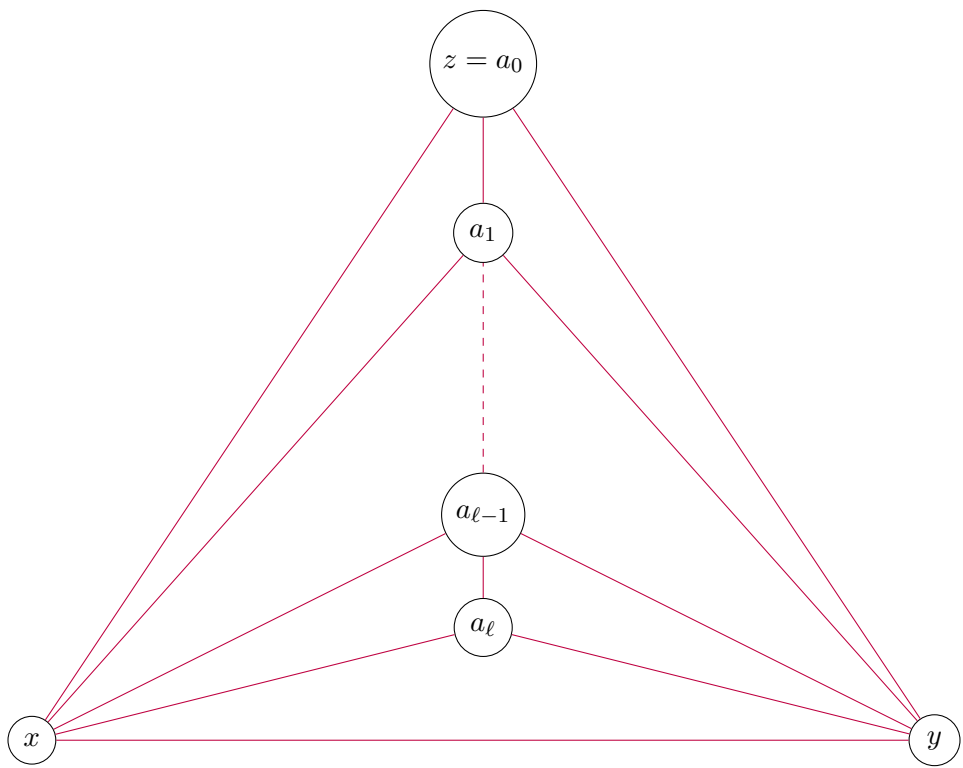
\end{defn}
\begin{rem}
If the path $a_0\to...\to a_\ell$ in an $\ell$-BHK filling (Definition \ref{defn:filling}) is simple, namely no vertex is visited more than once (as in the case of a good filling), then  an $\ell$-BHK filling of $xyz$ induces a filling of $xyz$ of volume $2\ell+1$ (see Figure \ref{fig:BHK_filling}). As every path can be simplified by removing cycles, every $\ell$-BHK filling induces a filling of $xyz$ of volume at most $2\ell+1$.
\end{rem}
The following few Claims aim to show that when $p=n^{-\nicefrac{1}{2}+\eta}$, the number of good $\ell$-BHK fillings in $\cY\sim Y(n,p)$ is well understood (for specific values of $\ell$), and that every triangle $\Delta\in \cY(2)$ participates in approximately the same number of such good fillings. 

\begin{claim}[Number of good $\ell$-BHK fillings of a specific triangle]\label{claim:number_of_BHK_fillings}
Let $\eps,\eta>0$, $p= n^{-\nicefrac{1}{2}+\eta}$, $\cY\sim Y(n,p)$ and $\ell=\lceil \nicefrac{1}{4\eta} \rceil$.  Then, for every  $x\neq y\neq z\in [n]$, the number of good $\ell$-BHK fillings of $\pi=xyz$ in $\cY$ is a.a.s.\ $(3+o(1))n^{(2\ell+1)\eta-\nicefrac{1}{2}}$.
\end{claim}

\begin{proof}
    This proof is  a quantitative analogue of Lemma 2.2 in \cite{babson2011fundamental}.

    Let $\cY_x\cap \cY_y$ be the intersection of appropriate links in $\cY$, and $\cY_{xy}$ the link of the edge $xy$. 
It is straightforward that the intersection of links $\cY_x\cap \cY_y$ is distributed as an Erd{\H{o}}s--R{\'e}nyi graph $G(n-2,p^2)=G(n-2,n^{-1+2\eta})$ for every $x\neq y\in [n]$. By Claim \ref{claim:expansion_Gnp}, Remark \ref{rem:unique_sphere_bound}  and using a union bound, we can assume that for every $x\neq y\neq z\in [n]$, the unique sphere $US_\ell(z)$ of radius $\ell$ around $z$ in $\cY_x\cap \cY_y$  contains $(1+o(1))n^{2\eta \ell}$ many vertices.
By Corollary \ref{cor:basic_prop}, we can assume that $|\cY_{xy}|$, the size of the link of the edge $xy$, is $(1+o(1))n^{\nicefrac{1}{2}+\eta}$. Now, $\cY_x\cap\cY_y$ and $\cY_{xy}$ are \textbf{independent of each other}, and hence the intersection of the unique sphere of radius $\ell$ around $z$ in $\cY_x\cap \cY_y$ with the vertices in $\cY_{xy}$ satisfies the conditions of Claim \ref{claim:intersection_of_independent_sets}. Therefore, using a union bound, for every $x\neq y\neq z$, a.a.s.\ this intersection is of size $(1+o(1))n^{(2\ell+1)\eta-\nicefrac{1}{2}}$. As every vertex in this intersection plays the role of $a_\ell$ in a good $\ell$-BHK filling of $xyz$  with base $xy$, we showed that the number of such fillings is $(1+o(1))n^{(2\ell+1)\eta-\nicefrac{1}{2}}$. Since every triangle $xyz$ has 
    $3$ potential bases, the total number of good $\ell$-BHK fillings of it is $(3+o(1))n^{(2\ell+1)\eta-\nicefrac{1}{2}}$, as claimed.
\end{proof}

\begin{claim}[Upper bound on the number of good $\ell$-BHK fillings each triangle participates in]\label{claim:upper_bound_fillings_abc}
    Let $\eps,\eta>0$, $p=n^{-\nicefrac{1}{2}+\eta}$, $\cY\sim Y(n,p)$ and $\ell=\lceil \nicefrac{1}{4\eta} \rceil$.  Then a.a.s, for every  $a\neq b\neq c\in [n]$, the number of good $\ell$-BHK fillings in $\cY$ that $\pi=abc$ participates in is at most  $(6\ell+3+o(1))n^{2\eta\ell}$.
\end{claim}

\begin{proof}
    Recall Definition \ref{defn:filling}, and in particular the notion of a base triangle and triangles of type $i$. To bound the number of fillings $abc$ participates in, it is enough to bound the number of such fillings $abc$ participates in at each type (and sum them up). 
    
    If the triangle $abc$ is a base type, then either $ab,bc$ or $ac$ are the bases. Since this is symmetric, let us count in how many of them $ab$ is the base. Every unique path of length $\ell$ in $\cY_a\cap\cY_b$ from $c$ defines an $\ell$-BHK filling for which $abc$ is the base. Hence, by Claim \ref{claim:expansion_Gnp} and Remark \ref{rem:unique_sphere_bound}, the number of such fillings is $(1+o(1))(np^2)^\ell=(1+o(1))n^{2\eta\ell}$, and  $abc$ participates in $(3+o(1))n^{2\eta\ell}$ many $\ell$-BHK fillings as a base type triangle (a.a.s).

   In case the triangle $abc$ is of type $i$, we will construct outwards from $abc$ an $\ell$-BHK filling such that we cover all possible good ones (and will thus get an upper bound as required). Similar to the base case, there are symmetries to break, as either $a,b$ or $c$ are part of the filled up triangle. Let us assume $a$ is part of the filled up triangle, and that $bc$ plays the role of the edge $a_{i-1}a_i$ in the path $a_0\to...\to a_\ell$. There is another symmetry we need to break, as either $b$ or $c$ play the role of $a_{i-1}$ (this is an oriented edge). Assume $b$ plays the role of $a_{i-1}$. By Corollary \ref{cor:basic_prop}, the size of $\cY_{bc}$ is $(1+o(1))n^{\nicefrac{1}{2}+\eta}$.  For each $y\in \cY_{bc}$, $\cY_{a}\cap \cY_y$ is distributed as $G(n-2,p^2)$, and thus there are $(1+o(1))n^{2\eta(\ell-i)}$ many vertices $w$ in the unique sphere around $c$ in it. In total, we have $(1+o(1))n^{2\eta(\ell-i)+\nicefrac{1}{2}+\eta}$ many such pairs $(y,w)$, and by our choice they are all distinct. For each such pair, there is a $p=n^{-\nicefrac{1}{2}+\eta}$ chance that the triangle $ayw$ is included in $\cY$, and these are independent events. So, out of these, $(1+o(1))n^{2\eta(\ell-i+1)}$ induce a good $\ell-i$-BHK filling of $aby$. To complete such a $\ell-i$-BHK filling to an $\ell$-BHK filling, we just need to add a path of length $i-1$ from $b$ in $\cY_a\cap \cY_y$. As there are at most $(1+o(1))n^{2\eta(i-1)}$ many such paths, we get that there are at most $(1+o(1))n^{2\eta \ell}$ such fillings. All in all, $abc$ participates in at most $(6+o(1))n^{2\eta\ell}$ many good $\ell$-BHK fillings as a type $i$ triangle. 

   Combining the two previous bounds provides an upper bound of $(6\ell+3+o(1))n^{2\eta\ell}$-many good $\ell$-BHK fillings that $abc$ participates in, as required.

\end{proof}

\section{\textbf{The mid range: Soficity of hyperbolic groups}}\label{sec:res_fin_hyperbolic_groups}

\subsection{The sampled pair}
As mentioned in the introduction, we define a sampling procedure for a pair $\cY,\cZ$ such that with high probability they have some desired properties. The sampling procedure is as follows:
\begin{algorithm}[H]
\caption{Sampling $\cY,\cZ$}\label{alg:sample_pair}
\begin{algorithmic}[1]

\STATE Sample $\cY$ from $Y(n,p)$.
\STATE Choose a uniformly random triangle $\Delta\notin \cY(2)$.
\STATE{Add $\Delta$ to $\cY$ to get $\cZ$.}

\end{algorithmic}
\end{algorithm}

\begin{rem}
    By a standard coupling argument (cf.~Chapter 1.1 in \cite{frieze2016introduction}), as long as $p=\omega(n^{-3})$, we may assume that $\cZ$ is contained in a $Y(n,2p)$ complex.
\end{rem}

\begin{fact} [Lemmas 3.10 and 3.12 in \cite{babson2011fundamental}]
Let $\eta>0$. For every triplet $x\neq y\neq z\in [n]$, if $p=O(n^{-\nicefrac{1}{2}-\eta})$, then the path $x\to y\to z\to x$ is a.a.s.\ non-trivial in $\pi_1(\cY,x)$ for $\cY\sim Y(n,p)$.
\end{fact}

\begin{fact}[Theorem 1.5 in \cite{babson2011fundamental}]
    Let $\eta>0$.  If $p=O(n^{-\nicefrac{1}{2}-\eta})$, then  a.a.s.\ $\pi_1(\cY,*)$ is non-trivial and hyperbolic for $\cY\sim Y(n,p)$.
\end{fact}

\begin{rem}
    We do not need the definition of a hyperbolic group in this paper, rather just the above fact. For the definition of a hyperbolic group, see for example \cite{kapovich2000equivalence}.
\end{rem}

\begin{fact} [Theorem 1.1 in \cite{kahle2021spectral}]
    Let $\lambda>0$. If $p=\omega\left(\frac{\log n}{n}\right)$, then a.a.s.\ $\cY\sim Y(n,p)$ is a $\lambda$-local spectral expander.
\end{fact}
\begin{cor}\label{cor:prop_of_pair}
    Let $0<\eta<\nicefrac{1}{2}$ and $\lambda>0$. If  $p=n^{-1+\eta}$, then a.a.s.\ over sampled pairs $\cY,\cZ$ according to Algorithm \ref{alg:sample_pair},
    \begin{enumerate}
        \item \underline{Added triangle is non-trivial}: The  triangle $\Delta\in \cZ(2)\setminus \cY(2)$ is not trivial in $\pi_1(\cY,*)$.
        \item \underline{Hyperbolic}: The fundamental groups $\pi_1(\cY,*)$ and $\pi_1(\cZ,*)$ are hyperbolic.
        \item \underline{Local expanders}: The complexes $\cY$ and $\cZ$ are $\lambda$-local spectral expanders.
    \end{enumerate}
\end{cor}
\subsection{Proving Theorem \ref{thm:midrange}}
Let $\eta>0$, $p=n^{-1+\eta}$. Assume by contradiction that every hyperbolic group is sofic and that a.a.s.\ $h_1(\cY,\Sym)=\omega\left(\frac{1}{pn^3}\right)$. As we assumed that the lower bound on $h_1(\cY,\Sym)$ holds a.a.s., we can deduce that also $h_1(\cZ,\Sym)=\omega\left(\frac{1}{pn^3}\right)$ a.a.s.\  as well. Indeed, the total variation distance between $\cY$ and $\cZ$ is bounded by the distance between the numbers $M$ and $M+1$ 0of triangles in $\cY$ and $\cZ$ respectively, where $M$ is distributed binomially with $\binom n3$ trials and success probability $p$.
Whenever $h_1(\cY,\Sym)>0$, the complex $\cY$ is $\rho$-cocycle stable in the $1^{\rm st}$ dimension with  permutations coefficients, where $\rho(\eps)=\frac{\eps}{h_1(\cY,\Sym)}$. The same is true for $\cZ$. Thus, by Fact \ref{fact:cocycle_stable_implies_hom_stable}, $\pi_1(\cY,*),\pi_1(\cZ,*)$ are $\rho'$-homomorphism stable, where $\rho'(\eps)=n^2\rho(\eps)$. By clause $(2)$ of Corollary \ref{cor:prop_of_pair}, the fundamental groups  $\pi_1(\cY,*),\pi_1(\cZ,*)$ are  hyperbolic, and therefore by our assumption, sofic. By Proposition \ref{prop:sof+stable=res_fin}, since the fundamental groups $\pi_1(\cY,*),\pi_1(\cZ,*)$ are sofic and $\rho'$-homomorphism stable, they are residually finite. 

Let $\Delta=abc$ be the triangle added to $\cZ$ in Algorithm \ref{alg:sample_pair}. By clause $(1)$ of Corollary \ref{cor:prop_of_pair}, a.a.s.\ $\Delta$ is non-trivial in $\pi_1(\cY,*)$. Therefore, by residual finiteness of $\pi_1(\cY,*)$, there is a homomorphism $\alpha\colon \pi_1(\cY,*)\to \Sym(n)$ such that $d_h(\alpha(\Delta),\Id)=1$. Recall that by choosing a spanning tree $T$ in $G(\cY)=G(\cZ)$  we get a presentation of $\pi_1(\cY,*)$ whose generators are the edges outside $T$. Thus, $\alpha$ induces a $1$-cocycle on $\cY$ by letting $\alpha(e)=\Id$ if $e\in T$ and keeping it the same on the rest of the edges.

Since $G(\cY)=G(\cZ)$, their $1$-cochains are the same. Hence, $\alpha$ can be thought of as a $1$-cochain of $\cZ$ as well. The local defect of $\alpha$ as a $1$-cochain of $\cZ$ is
\[
\Vert \delta \alpha \Vert =\Ex_{[\Delta']\in \cZ(2)}[\|\delta\alpha(\Delta')\|].
\]
Since $\alpha$ was a $1$-cocycle of $\cY$, $\|\delta\alpha(\Delta')\|=0$ as long as $\Delta'\neq \Delta$, and $\|\delta\alpha(\Delta)\|=1$. Hence, $$\Vert \delta \alpha \Vert =\frac{1}{|\cZ(2)|}.$$

Let $\beta \colon \overrightarrow\cZ(1)\to \Sym(N)$ be a $1$-cocycle of $\cZ$. Then, it is a $1$-cocycle of $\cY$ which satisfies  $\beta(\Delta)=\Id$.\footnote{Note that $\delta\beta$ is only defined on triangles of $\cY$, and $\Delta$ is not a triangle of $\cY$. Regardless, as in Remark \ref{rem:maps_on_paths}, $\beta$ can be evaluated on any closed path, specifically the perimeter of the triangle $\Delta$.} By Lemma \ref{lem:hamming_dist_cochains_and_epxansion_in_covering}, $d_h(\alpha,\beta)\geq \frac{\Vert \delta f\Vert}{2\mu_0(D)}$, where $D=\{(y,i,i)\mid y\in \cY(0), i\in [n]\}$ and $f\colon \cY(0)\times [n]\times [N]\to \FF_2$ is the charateristic function of $D$, and the undelying graph is  $\cA$ the covering of $\cY$ associated with $\alpha\times \beta$. Let $\{C\mid C\in \pi_0(\cA)\}$ be the connected components of $\cA$. Denote by $f_C$ the restriction of $f$ to $C$, namely the characteristic function of $C\cap D$.

\begin{claim}\label{claim:D_is_sparse_in_connected_components}
  For every connected component $C$ of $\cA$, 
\[
\frac{|\mu_0(C\cap D)|}{|\mu_0(C)|}\leq \frac{1}{2}.
\]  
\end{claim}

\begin{proof}[Proof of Claim \ref{claim:D_is_sparse_in_connected_components}] Let $\varphi\colon \cA \to \cY$ be the covering map induced by $\alpha\times \beta$, namely $\varphi(y,i,j)=y$ for every $y\in \cY(0)$, $i\in[n]$ and $j\in [N]$. Then,
since $\mu_0$ is uniform on fibers $$\varphi^{-1}(y)=\{(y,i,j)\mid i\in[n],j\in[N]\},$$ it is enough to prove that $\frac{|\varphi^{-1}(y)\cap C\cap D|}{|\varphi^{-1}(y)\cap C|}\leq \frac{1}{2}$ for every $y\in \cY(0)$. Let $(y,i,i)\in \varphi^{-1}(y)\cap C\cap D$. Choose a path $\sigma$ from $y$ to $a$ --- the first vertex of $\Delta=abc$ ---  in the spanning three $T$. By the definition of $\cA$,  $(y,\alpha\times \beta (\sigma \Delta \bar\sigma).(i,i))$  is the origin point of the path $\sigma \Delta \bar\sigma$ whose endpoint is $(y,i,i)$, and hence they are in the same connected component $C$. On the other hand,
\[
\alpha\times \beta (\sigma \Delta \bar\sigma).(i,i)=(\alpha(\sigma\Delta\bar\sigma).i,\beta(\sigma\Delta\bar\sigma).i).
\]
Since $\beta(\Delta)=\Id$, its conjugation $\beta(\sigma\Delta\bar\sigma)$ is also $\Id$. On the other hand, $\sigma$ is contained in $T$ and thus $\alpha(\sigma)=\alpha(\bar\sigma)=\Id$. Hence, $\alpha(\sigma\Delta\sigma)=\alpha(\Delta)$. But $d_h(\alpha(\Delta),\Id)=1$, namely it has no fixed points. All in all,
\[
(\alpha(\sigma\Delta\bar\sigma).i,\beta(\sigma\Delta\bar\sigma).i)=(j,i)
\]
for some $j\neq i$. This finishes the proof, since the map $(y,i,i)\to (y,j,i)\notin D$ we just defined is injective and preserves the connected component (and fiber). 
\end{proof}

Back to the proof of Theorem \ref{thm:midrange}. We have
\[
\|\delta f\|=\Pro_{e\sim \mu_1}[f(\iota(e))\neq f(\tau(e))]=\Ex_{C\in \pi_0(\cA)} \underbrace{\Pro_{e\in C}[f_C(\iota(e))\neq f_C(\tau(e))]}_{\Vert \delta f_C\Vert},
\]
where $\Ex_{c\in \pi_0(\cA)}$ is according to the uniformly lifted $\mu_0$ and $\Pro_{e\in C}$ is according to the uniformly lifted $\mu_1$. By Claim \ref{claim:D_is_sparse_in_connected_components},  $d_h(f_C,Z^0(C,\FF_2))=\nicefrac{\mu_0(C\cap D)}{\mu_0(C)}$, as $Z^0(C,\FF_2)$ consists of the constant functions, and by the aforementioned claim, $f_C$ is closer to the constant zero function than to the all one function. Furthermore, by combining clause $(3)$ of Corollary \ref{cor:prop_of_pair} together with Corollary \ref{cor:local_exp_implies_exp_of_coverings}, we get $\frac{\Vert \delta f_C\Vert}{\nicefrac{\mu_0(C\cap D)}{\mu_0(C)}}\geq h_0(C,\FF_2)\geq \frac{1-2\lambda}{1-\lambda}$.
Therefore,
\[
\Ex_{C\in \pi_0(\cA)} [\Vert \delta f_C\Vert]\geq \sum_{C\in \pi_0(\cA)}\mu_0(C\cap D)\frac{1-2\lambda}{1-\lambda}=\mu_0(D)\frac{1-2\lambda}{1-\lambda}.
\]
All in all, 
\[
d_h(\alpha,\beta)\geq \frac{\Vert \delta f\Vert}{2\mu_0(D)}\geq \frac{1-2\lambda}{2-2\lambda}.
\]
But this means that the global defect of $\alpha$ as a $1$-cochain of $\cZ$ is at least $\frac{1-2\lambda}{2-2\lambda}$, which we can choose to be larger than $\frac{1}{3}$. Hence $h_1(\cZ,\Sym)\leq \frac{3}{|\cZ(2)|}$. But, by Corollary \ref{cor:basic_prop}, a.a.s.\ $|\cZ(2)|\geq \frac{p\binom{n}{3}}{2}$. Therefore $h_1(\cZ,\Sym)=O\left(\frac{1}{pn^3}\right)$, which contradicts our assumption that $h_1(\cZ,\Sym)=\omega\left(\frac{1}{pn^3}\right)$.
\qed

\section{\textbf{The trivial fundamental group regime}}\label{sec:high_range}

Before moving to the main proposition, we need the following fact. 

\begin{fact}[Lemma 3.1 in \cite{CL_part2}]\label{lem:identity_on_spanning_tree}
    Let $\cX$ be a simplicial complex, let $T$ be a spanning tree in $G(\cX)$ and let $*\in \cX(0)$. Then, for every $1$-cochain $\alpha\colon \overrightarrow\cX(1)\to \Sym(n)$, there is a $0$-cochain $\beta\colon \cX(0)\to \Sym(n)$ such that $\beta(*)=\Id$ and $\beta.\alpha(e)=\Id$ for every $e\in T$.
    \end{fact}
    
\begin{prop}\label{prop:positive_Cheeger_high_p}
    Let $\eta>0$, $p=n^{-\nicefrac{1}{2}+\eta}$ and $\cY\sim Y(n,p)$. Then asymptotically almost surely, $$h_1(\cY,\Sym)\geq \frac{\eta}{42}.$$
\end{prop}
\begin{rem}
    Our proof is in the spirit of  what is known in the literature as the ``cones method'' (cf.\ \cites{gromov2010singularities,lubotzky2016expansion,kozlov2019quantitative,kaufman2019coboundary,dikstein2023coboundary}).
\end{rem}
\begin{proof}

Let $\alpha\colon \overrightarrow\cY(1)\to \Sym (N)$ be a $1$-cochain of $\cY\sim Y(n,p)$. By Fact \ref{lem:identity_on_spanning_tree}, for every $x\in [n]$ we can choose a $0$-cochain $\beta_x$ such that 
\[
\forall y\in [n]\ \colon\ \ \beta_x.\alpha(xy)=\Id. 
\]
Also, by using the characterization in \eqref{eq:char_of_dist_to_coboundaries} and the fact that $B^1(\cY,\Sym)=Z^1(\cY,\Sym)$ when $\pi_1(\cY,*)$ is trivial, we have $$d_h(\alpha,Z^1(\cY,\Sym))\leq \min_x\Vert\beta_x.\alpha\Vert \leq \Ex_x[\Vert\beta_x.\alpha\Vert ],$$ where $\Ex_x$ is the uniform distribution.  
Now, using the notation defined in Remark \ref{rem:maps_on_paths},
\[
\begin{split}
    \Ex_x[\Vert\beta_x.\alpha\Vert ]&=\Ex_{x}\Ex_{yz\sim \mu_1}[d_h(\beta_x.\alpha(yz),\Id)]\\
    &= \Ex_{x}\Ex_{yz\sim \mu_1}[d_h(\underbrace{\alpha(xy)\alpha(yz)\alpha(zx)}_{=\alpha(xyz)},\Id)]\\
    &=\frac{1}{n}\sum_{x\in[n]}\sum_{y\neq z\in[n]}\mu_1(yz)d_h(\alpha(xyz),\Id)\\
    &=\frac{1}{n}\sum_{x\neq y\neq z\in[n]}\mu_1(yz)d_h(\alpha(xyz),\Id),
\end{split}
\]
where the last equality is since $x=y$ or $x=z$ implies that $\alpha(xy)\alpha(yz)\alpha(zx)=\Id$.
Note that by the triangle inequality and the bi-invariance of the Hamming metric, if $xyz$ is filled (Definition \ref{defn:filling}) by the triangles $\Delta_1,...,\Delta_k$, then
\[
\begin{split}
 d_h(\alpha(xyz),\Id)&\leq \sum_{i=1}^k d_h(\alpha(\Delta_i),\Id).
\end{split}
\]
Therefore, $d_h(\alpha(xyz),\Id)\leq \displaystyle{\min_{\Delta_1,...,\Delta_k\ \textrm{fills}\ xyz}}\sum_{i=1}^k d_h(\alpha(\Delta_i),\Id)\leq \Ex \left[ \sum_{i=1}^k d_h(\alpha(\Delta_i),\Id)\right]$, where $\Ex$ is any distribution over fillings of $xyz$. In particular, choosing $\ell=\lceil \nicefrac{1}{4\eta} \rceil$, we have
\[
\begin{split}
    &d_h(\alpha,Z^1(\cY,\Sym))\leq\\
    &\frac{1}{n}\sum_{x\neq y\neq z\in [n]}\mu_1(yz)\Ex_{\Delta_1,...,\Delta_{2\ell+1}\ \textrm{good}\ \ell-\textrm{BHK\ filling\ of}\ xyz}\left[ \sum_{i=1}^{2\ell+1} d_h(\alpha(\Delta_i),\Id)\right],
\end{split}
\]
where the distribution over such good $\ell$-BHK fillings is the uniform one. By Claim \ref{claim:number_of_BHK_fillings}, there are at least $(3+o(1))n^{(2\ell+1)\eta-\nicefrac{1}{2}}$ many good $\ell$-BHK fillings of $xyz$, and thus 
\[
\begin{split}
&\Ex_{\Delta_1,...,\Delta_{2\ell+1}\ \textrm{good}\ \ell-\textrm{BHK\ filling\ of}\ xyz}\left[ \sum_{i=1}^{2\ell+1} d_h(\alpha(\Delta_i),\Id)\right]\leq \\
&\frac{1}{(3+o(1))n^{(2\ell+1)\eta-\nicefrac{1} {2}}}\sum_{\Delta_1,...,\Delta_{2\ell+1}\ \textrm{good}\ \ell-\textrm{BHK\ filling\ of}\ xyz}\left[ \sum_{i=1}^{2\ell+1} d_h(\alpha(\Delta_i),\Id)\right].
\end{split}
\]
By Corollary \ref{cor:basic_prop}, a.a.s.\ $\mu_1(yz)\leq {(2+o(1)})n^{-2}$ for all $y\neq z\in [n]$, and hence
\[
\begin{split}
    &d_h(\alpha,Z^1(\cY,\Sym))\leq\\
    &\frac{2}{(3+o(1))n^{(2\ell+1)\eta-\nicefrac{1} {2}}\cdot n^3}\sum_{x\neq y\neq z\in [n]}\sum_{\Delta_1,...,\Delta_{2\ell+1}\ \textrm{good}\ \ell-\textrm{BHK\ filling\ of}\ xyz}\left[ \sum_{i=1}^{2\ell+1} d_h(\alpha(\Delta_i),\Id)\right].
\end{split}
\]
By Claim \ref{claim:upper_bound_fillings_abc}, a specific $abc\in \cY(2)$ appears at most  $(6\ell+3+o(1))n^{2\eta\ell}$ many times in the above sum, and hence
\[
\begin{split}
d_h(\alpha,Z^1(\cY,\Sym))&\leq\frac{(12\ell+6+o(1))n^{2\eta\ell}}{(3+o(1))n^{(2\ell+1)\eta-\nicefrac{1} {2}}\cdot n^3}\sum_{abc\in \cY(2)}d_h(\alpha(abc),\Id)\\
&\leq (24\ell+12+o(1))\cdot \frac{\sum_{abc\in \cY(2)}d_h(\alpha(abc),\Id)}{\nicefrac{pn^3}{6}}.
\end{split}
\]
By Corollary \ref{cor:basic_prop}, $|\cY(2)|=(1+o(1))\frac{pn^3}{6}$, and we have 
\[
d_h(\alpha,Z^1(\cY,\Sym))\leq (24\ell+12+o(1))\|\delta\alpha\|.
\]
As $24\ell+12\leq \frac{6}{\eta}+36$, and $36<\frac{36}{\eta}$, we conclude that $h_1(\cY,\Sym)\geq \frac{\eta}{42}$.
\end{proof}

\section{{\textbf{The subcritical regime}}}\label{sec:low_range}

In this section we study $h_1(\cY,\Sym)$ where $\cY\sim Y(n,p)$ and $p=n^{-1-\eta}$ for some fixed $\eta>0$. In this regime, the simplicial complex $Y(n,p)$ has a well-understood {\em subcritical} structure: it is comprised of small strongly-connected components, that are collapsible to graphs ~\cite{farberDCG}.

\subsection{Strongly-connected components, collapsibility, and Cheeger constant}
We start by a general study of collapsible strongly-connected components, and their Cheeger constants. 

Given a $2$-dimensional simplicial complex $X$, its dual graph is a graph whose vertices are the triangles of $X$, and two triangles are adjacent if they share an edge. The complex $X$ is called {\em strongly-connected} if its dual graph is connected. More generally, a subcomplex of $X$ which is comprised of all the triangles in a connected component of the dual graph, and the vertices and edges they contain, is called a strongly-connected component. Clearly, every $2$-complex $X$ has an edge-disjoint decomposition to strongly-connected components and isolated edges and vertices. 

The  Cheeger constant $h_1(X,\Gamma)$ can be read off the Cheeger constants of the strongly-connected components of $X$.

\begin{lem}\label{lem:h1_min}
    Let $X$ be a $2$-complex and $\cC$ the set of its strongly-connected components. Then,
    \[
    h_1(X,\Gamma) = \min_{C\in\cC}h_1(C,\Gamma)\,.
    \]
\end{lem}
\begin{proof}
    Let $\mu$ be the measure on $\cC$ given by sampling a triangle from $X$ by the uniform measure $\mu_2$ and selecting its component. Note that since $\mu_1$ is descending from $\mu_2$ then $\mu$ is also obtained by sampling an edge by the measure $\mu_1$ and selecting its component.
    
    For an $i$-cochain $\alpha$ of $X$, we denote by $\alpha_C$ its restriction to a strongly-connected component $C\in\cC$. A direct computation yields that 
    \(
    d(\alpha,\beta)=\Ex_{C\sim\mu}[d(\alpha_C,\beta_C)] 
    \)
     for every two $i$-cochains $\alpha,\beta$. Therefore,
    $\|\delta\alpha\| = \Ex_{C\sim\mu}[\|\delta\alpha_C\|]$ for every $1$-cochain $\alpha$, and, in addition, $d(\alpha,Z^1(X,\Gamma))=\Ex_{C\sim\mu}[d(\alpha_C,Z^1(C,\Gamma))]$, since the closest $1$-cocycle to $\alpha$ is obtained by choosing the closest $1$-cocycle to $\alpha_C$ in every component $C$.

    In conclusion, for every $1$-cochain $\alpha$,
    \[
    \|\delta \alpha\|\ge \min_{C\in\cC}h_1(C,\Gamma)\cdot \Ex_{C\sim\mu}[d(\alpha_C,Z^1(C,\Gamma))]=\min_{C\in\cC}h_1(C,\Gamma)d(\alpha,Z^1(X,\Gamma))]\,,
    \]
    whence $h_1(X,\Gamma)\ge \min_{C\in\cC}h_1(C,\Gamma).$ 
    
    On the other hand, let $C_0$ be a component attaining the minimum $h_1(C,\Gamma)$, using the $1$-cochain $\alpha_0\in C^1(C_0,\Gamma)$. Extend $\alpha_0$ to a $1$-cochain $\alpha$ on $X$ by setting $\alpha(e)=\Id$ for every $e\notin C_0$. Then, 
    \[
    \|\delta\alpha\|=\mu(C_0)\cdot\| \delta\alpha_0\|=\mu(C_0)\cdot h_1(C_0,\Gamma)\cdot d(\alpha_0,Z^1(C_0,\Gamma))=h_1(C_0,\Gamma)\cdot d(\alpha,Z^1(X,\Gamma))\,,
    \]
    whence $h_1(X,\Gamma)\le h_1(C_0,\Gamma),$ which concludes the proof. 
\end{proof}

Collapsibility is one of the most well-studied concepts in the theory of random simplicial complexes. An edge $e$ in a simplicial complex $X$ is called {\em free} if it is contained in a unique triangle $\Delta$ of $X$. In such a case, the removal $X\searrow X\setminus\{e,\Delta\}$ of $e$ and $\Delta$ from $X$ is called an elementary collapse. If all the triangles in $X$ can be removed in a sequence of elementary collapses, we say that $X$ is collapsible to a graph. 

Observe that every strongly-connected $2$-complex $C$ satisfies $|C(2)|+2\ge |C(0)|$. Indeed, explore its dual graph by some graph-search algorithm. Initially, one triangle and $3$ vertices are exposed. Afterwards, in every step the algorithm exposes one additional triangle and at most one new vertex. The following claim discusses the collapsibility of strongly-connected $2$-complexes with a few faces. 

\begin{claim}\label{clm:simple_implies_collapsible}
    Let $C$ be a strongly-connected component satisfying $|C(0)|>|C(2)|$. Then, $C$ is collapsible to a graph $G$ that intersects the boundary of every triangle in $C$.
\end{claim}
\begin{proof}
    We construct a sequence of elementary collapses that establishes the claim. Suppose that in graph-search algorithm on the dual-graph of $C$ the triangles are exposed in the order $\Delta_1,...,\Delta_f$.  The fact that $|C(2)|<|C(0)|$ implies that each step exposes a new vertex, except perhaps one step.
    
    \noindent\textbf{Case 1.} If $|C(0)|=|C(2)|+2$ then in every step $j$, 2 new edges of $\Delta_j$ are exposed (in the first step $3$ new edges are exposed, but we consider only 2 of them). We arbitrarily select one of them, denote it $e_j$, and we denote the other new edge by $\tilde e_j$. We collapse $C$ in the reverse order to which it was exposed, using $e_j$ as the free edge of $\Delta_j$. Note that $\tilde e_j$ has not been collapsed, thus it belongs to the resulting graph $G$, and so $G$ intersects the boundary of every triangle $\Delta_j$ in $C$.

    \noindent\textbf{Case 2.} If $|C(0)|=|C(2)|+1$ then there exists a unique special step $s$ in which no new vertex has been exposed. If $2$ new edges have been exposed in this step, then the same construction in Case 1 works. Otherwise, note that $1$ new edge have been exposed in step $s$. Indeed, otherwise the three boundary edges of $\Delta_s$ belong to the subcomplex with triangles $\Delta_1,...,\Delta_{s-1}$, but this cannot occur since in every step $1\le j <s$ a new vertex is exposed. In such a case, we amend our construction as follows: Start by letting $e_s$ be the unique newly exposed edge in step $s$, and arbitrarily select $\tilde e_s$ from the other edges in the boundary of $\Delta_s$. Then, sequentially, for each $j< s$ we select a newly exposed edge in step $j$ that is not $\tilde e_s$ as $e_j$, and the other newly exposed edge in this step as $\tilde e_j$. This can always be done since there are two new edges and only one edge, $\tilde e_s$, we may need to avoid when choosing $e_j$. For $j>s$, we select $e_j,\tilde e_j$ arbitrarily from the new edges in step $j$. 

    As before, we collapse $C$ in the reverse order to which it was exposed, using $e_j$ as the free edge of $\Delta_j$. This is possible since $e_j$ is a new edge step $j$ for every $1\le j \le f$. In addition, no $\tilde e_j$ has been collapsed, thusthe resulting graph $G$ intersects the boundary of every triangle $\Delta_j$ in $C$.
    \end{proof}

We now give a lower bound for the Cheeger constant of such complexes
\begin{lem}\label{lem:h1_col}
    If $C$ is a strongly-connected $2$-complex with $f$ triangles that is collapsible to a graph that intersects the boundary of every triangle in $C$. Then $h_1(C,\Gamma)\ge 1/f\,.$
\end{lem}
\begin{proof}
    Let $\alpha\in C^1(C,\Gamma)$. Suppose that $C$ admits a sequence of elementary collapses
\[
C=C^{(m)}\searrow C^{(m-1)}\searrow\cdots\searrow C^{(1)}\,,
\]
where $C^{(1)}$ is a graph that intersects the boundary of every triangle in $C$. We use the collapse sequence to inductively construct a $1$-cocycle $\beta$ that is close to $\alpha$. First, for every edge $e\in C^{(1)}$ we set $\beta(e)=\alpha(e)$. For every $1<j\le m$, denote by $e_j=xy,~\Delta_j=xyz$ the edge and triangle that are removed in the elementary collapse $C^{(j)}\searrow C^{(j-1)}\,.$ Assume, by induction, that $\beta(e)$ is already defined for every $e\in C^{(1)}\cup\{e_2,...,e_{j-1}\}$. In particular, $\beta(xz),\beta(yz)$ are already defined, and we define $\beta(xy)=\beta(xz)\beta(zy)$ so that $\delta\beta(\Delta_j)=\Id$. Therefore, by the triangle inequality,
\[
d(\alpha(xy),\beta(xy)) \le 
d(\alpha(xz),\beta(xz))+
d(\alpha(yz),\beta(yz))
+d(\delta\alpha(\Delta_j),\Id).
\]
Since $C^{(1)}$ intersects the boundary of every triangle in $C$, we assume, without loss of generality, that $xz\in C^{(1)},$ hence $\alpha(xz)=\beta(xz)$. Therefore, a straightforward induction gives
\[
d(\alpha(e_j),\beta(e_j))\le \sum_{i \ge j}d(\delta\alpha(\Delta_j),\Id).
\]
In consequence, since $\beta \in Z^1(C,\Gamma)$, we find
\[
d(\alpha,Z^1(C,\Gamma)) \le \sum_{\Delta\in C}d(\delta\alpha(\Delta),\Id) = f\cdot \|\delta\alpha\|\,,
\]
as claimed.
\end{proof}

The following example shows that the lower bound from Lemma \ref{lem:h1_col} is of the right order.

\begin{example}\label{exm:upper}
    Let $C$ be a $2$-complex comprised of a triangle $xyz$ and for each edge of delta, $m$ additional triangles (which we call wings) attahced to it: $xyv_i,xzu_i,yzw_i$ for $i=1,...,m$. In particular, $|C(0)|=3+3m$ and $ |C(2)|=1+3m$. Let $\Gamma$ be a group and $\sigma\in\Gamma$ of distance $\varepsilon>0$ from $\Id$. Let $\alpha\in C^1(C,\Gamma)$ defined by $\alpha(xy)=\alpha(xv_1)=\cdots=\alpha(xv_m)=\sigma$ and $\Id$ elsewhere. Then, $\delta\alpha(xyz)=\sigma$ and $\Id$ for every other triangle. In consequence,
    \(
    \|\delta\alpha\| = \varepsilon/(3m+1)\,.
    \)
    Let $\beta \in Z^1(C,\Gamma)$. Denote by $d_e:=d(\alpha(e),\beta(e))$. Then,
    \begin{equation}\label{eq:dists}
    d(\alpha,\beta) = \frac{m+1}{3(3m+1)}(d_{xy}+d_{yz}+d_{zx}) + \frac{1}{3(3m+1)}\sum_{i=1}^{m}(d_{xv_i}+d_{v_iy}+d_{yw_i}+d_{w_iz}+d_{zu_i}+d_{u_ix})\,.    
    \end{equation}
    Note that  the cycles $xyz$ and $xv_iyw_izu_i$, for every $i=1,...,m$ are all null-homotopic in $C$. Therefore, the product of $\beta(e)$ over the edges of each of these cycles is $\Id$. In addition, the product of $\alpha(e)$ over the edges of each of these cycles is $\sigma$. By the triangle inequality we find the each sum of distances in \eqref{eq:dists} is bounded from below by $d(\sigma,\Id)=\varepsilon$. Therefore,
    \[
    d(\alpha,\beta)\ge \frac{2m+1}{3(3m+1)}\varepsilon.
    \]
    We deduce that $h_1(C,\Gamma)\le  3/(2m+1).$
\end{example}

\subsection{Subcritical Linial-Meshulam Complexes}
Consider $\cY\sim Y(n,p)$, where $p=n^{-(1+\eta)}$ for $\eta>0$. In this regime, all the strongly-connected components of $\cY$ are typically small and structurally simple.
\begin{claim}\label{clm:comp_size}
    Fix $\eta>0$, and let $\cY\sim Y(n,p)$, where $p=n^{-1-\eta}$. Then, a.a.s.\, every strongly-connected component $C$ of $\cY$ satisfies 
    $|C(2)|\le  2/\eta$.
\end{claim}

\begin{proof}
    Set $f$ to be the smallest integer greater than $2/\eta$. Let $C$ be a strongly-connected complex with $|C(2)|\ge f$. explore its dual graph by some graph-search algorithm. Initially, one triangle and $3$ vertices are exposed. Afterwards, in every step the algorithm exposes one additional triangle and at most one new vertex. Terminate the search when $f$ triangles are exposed, and note that at this point $v\le f+2$ are exposed. Therefore, the event that $\cY$ has a strongly-connected component $C$ with  $|C(2)|\ge   f$ is contained in the event that some $v\le f+2$ vertices in $\cY$ span $f$ triangles. By the union-bound, this occurs with probability at most

    \[
\sum_{3\le v\le f+2}\binom nv\binom{\binom v3}{v}p^f=O\left(n^{f+2-(1+\eta) f}\right)\to 0\,,
    \]
    as $n\to\infty$, as claimed.
\end{proof}

\begin{claim}\label{clm:comp_simple}
    Fix $\eta>0$, and let $\cY\sim Y(n,p)$, where $p=n^{-1-\eta}$. Then, a.a.s.\, every strongly-connected component $C$ of $\cY$ satisfy 
    $|C(0)|>|C(2)|.$
\end{claim}
\begin{proof}
    By the previous claim, we may assume that $|C(2)|\le 2/\eta$. The probability that there exist integers $4\le v\le f \le 2/\eta$, and $v$ vertices that span $f$ triangles in the $\cY$ is at most
    \[
    \sum_{4\le v\le f \le 2/\eta} \binom nv\binom {\binom v3} vp^f = O(n^{-4\eta}) \to 0\,,
    \]
as $n\to\infty$, as claimed.
\end{proof}

These two claims will let us prove a lower bound for $h_1(\cY,\Gamma)$. For the upper bound we will show that $\cY$ a.a.s.\  contains the strongly-connected component that appears in Example \ref{exm:upper}.

\begin{claim}\label{clm:wings_exist}
    Fix $2>\eta>0$, and let $\cY\sim Y(n,p)$, where $p=n^{-1-\eta}$. Then, a.a.s.\, $\cY$ has a strongly-connected component that is isomorphic to the complex $C$ in Example \ref{exm:upper} with $m\ge 0$ wings, provided $3m+1<2/\eta$...0
\end{claim}
\begin{proof}
    There are $N=\binom {n}{3,m,m,m,n-3m-3}$ different subcomplexes that are isomorphic to $C$. Denote them by $C_1,...,C_N$ and let $E_i$ be the event that $C_i$ is a strongly-connected component of $\cY$. Therefore, the expectation of the number $X$ of strongly-connected components isomorphic to $C$ is at least
    \[
    \Ex[X]\ge Np^{3m+1}(1-p)^{(6m+3)n}\,.
    \]
    since, for each $1\le i \le N$, each of the $6m+3$ edges in $C_i$ is contained in  most $n$ triangles, that we need to exclude from $\cY$ so that the $3m+1$ triangles in $C_i$ will form an entire component. In addition, note that $E_i\cap E_j$, for $i\ne j$, can occur  only if $C_i$ and $C_j$ are edge disjoint, since otherwise the corresponding components would have coalesce. In particular $\Pro(E_i\cap E_j)>0$ implies that $C_i,C_j$ and triangle disjoint, and therefore $\Pro(E_i\cap E_j)\le p^{6m+2}.$ We obtain that
    \[
    \Ex[X^2]\le \sum_{i=1}^{N}\Pro(E_i) + N(N-1)p^{6m+2} \le \Ex[X] + (Np^{3m+1})^2\,,
    \]
    By Checbyshev's inequality,
    \[
    \Pro(X=0) \le \frac{\mathrm{Var}(X)}{\Ex[X]^2}\le \frac{1}{\Ex[X]}+(1-p)^{-2(6m+3)n} -1\,.
    \]
    Note that $(1-p)^{(6m+3)n}\ge e^{-2p(6m+3)n}\to 1$ as $n\to\infty$, whence $\Ex[X]=\Omega(n^{2-\eta(3m+1)})\to\infty\,,$ by our assumption on $m$. Therefore, $X>0$ a.a.s., as claimed.
\end{proof}

\begin{proof}[Proof of Theorem \ref{thm:low_range}]
We start with the lower bound. Let $C$ be a strongly-connected componet of $\cY$. By Claims \ref{clm:comp_simple}, we have that a.a.s., $|C(0)|>|C(2)|$. By Claim \ref{clm:simple_implies_collapsible} and Lemma \ref{lem:h1_col} we deduce that $h_1(C,\Gamma)\ge 1/|C(2)|$. Combined with Claim \ref{clm:comp_size}, we find that a.a.s., $h_1(C,\Gamma)\ge \eta/2$ for every  strongly-connected componet $C$ of $\cY$. The lower bound in the theorem is derived by Lemma \ref{lem:h1_min}.

For the upper bound, Lemma \ref{lem:h1_min} asserts that we need to show that a.a.s.\ there exists a strongly-connected component $C$ in $\cY$ such that $h_1(C,\Gamma)\le 9\eta/4+3$, which can be derived from Claim \ref{clm:wings_exist} and Example \ref{exm:upper}. Indeed, if $2>\eta\ge 1/2$ the claim guarantees a strongly-connected component with one triangle whose cocycle Cheeger constant is $3$. And if $\eta < 1/2$ the claim asserts the existence of a strongly-connected component isomorphic to $C$ from the example with $m\ge (2/\eta -1)/3-1$, whence $h_1(C,\Gamma)\le 9\eta/4+3$.
\end{proof}

\bibliographystyle{plain}
\bibliography{Bib}

\end{document}